\documentclass[11pt,letterpaper]{amsart}

\usepackage[margin=1.08in]{geometry}
\usepackage{amsmath,amssymb,amsthm,mathtools}
\usepackage{microtype}
\usepackage[T1]{fontenc}
\usepackage{libertinus}
\usepackage{enumitem}
\usepackage{booktabs}
\usepackage{xcolor}
\definecolor{black}{RGB}{0,0,0}
\usepackage{array}
\usepackage{tikz}
\usetikzlibrary{decorations.pathreplacing,arrows.meta,positioning}
\usepackage{cite}
\usepackage{hyperref}
\usepackage{aliascnt}
\usepackage[nameinlink,capitalize]{cleveref}

\hypersetup{
  colorlinks=true,
  linkcolor=black,
  citecolor=black,
  urlcolor=black,
  pdftitle={The Bombieri--Weyl Contractivity Threshold},
  pdfauthor={Daniel Nunez-Alarcon, Daniel M. Pellegrino, Anselmo Raposo Jr., and Eduardo V. Teixeira}
}
\setlist{itemsep=0.25em,topsep=0.45em}
\numberwithin{equation}{section}
\allowdisplaybreaks

\newaliascnt{proposition}{theorem}
\newtheorem{proposition}[proposition]{Proposition}
\aliascntresetthe{proposition}

\newaliascnt{lemma}{theorem}
\newtheorem{lemma}[lemma]{Lemma}
\aliascntresetthe{lemma}

\newaliascnt{corollary}{theorem}
\newtheorem{corollary}[corollary]{Corollary}
\aliascntresetthe{corollary}

\newtheorem{maintheorem}{Theorem}

\theoremstyle{definition}
\newaliascnt{definition}{theorem}

\aliascntresetthe{definition}

\theoremstyle{remark}
\newaliascnt{remark}{theorem}

\aliascntresetthe{remark}

\crefname{theorem}{theorem}{theorems}
\Crefname{theorem}{Theorem}{Theorems}
\crefname{proposition}{proposition}{propositions}
\Crefname{proposition}{Proposition}{Propositions}
\crefname{lemma}{lemma}{lemmas}
\Crefname{lemma}{Lemma}{Lemmas}
\crefname{corollary}{corollary}{corollaries}
\Crefname{corollary}{Corollary}{Corollaries}
\crefname{definition}{definition}{definitions}
\Crefname{definition}{Definition}{Definitions}
\crefname{remark}{remark}{remarks}
\Crefname{remark}{Remark}{Remarks}
\crefname{maintheorem}{theorem}{theorems}
\Crefname{maintheorem}{Theorem}{Theorems}

\newcommand{\NN}{\mathbb{N}}
\newcommand{\CC}{\mathbb{C}}
\newcommand{\TT}{\mathbb{T}}
\newcommand{\DD}{\mathbb{D}}
\newcommand{\e}{\mathrm e}
\newcommand{\supp}{\operatorname{supp}}
\newcommand{\pat}{\operatorname{pat}}

\title[The Bombieri--Weyl contractivity threshold]
{The Bombieri--Weyl Contractivity Threshold}

\author[D. N\'u\~nez-Alarc\'on]{Daniel N\'u\~nez-Alarc\'on}
\address{Department of Mathematics, Universidad Nacional de Colombia, Bogot\'a, Colombia}
\email{Dnuneza@unal.edu.co}
\author[D. M. Pellegrino]{Daniel M. Pellegrino}
\address{Department of Mathematics, Universidade Federal da Para\'iba, Jo\~ao Pessoa, PB, Brazil}
\email{daniel.pellegrino@academico.ufpb.br}
\author[A. Raposo Jr.]{Anselmo Raposo Jr.}
\address{Department of Mathematics, Federal University of Maranh\~ao, S\~ao Lu\'is, MA, Brazil}
\email{anselmo.junior@ufma.br}
\author[E. V. Teixeira]{Eduardo V. Teixeira}
\address{Department of Mathematics, Oklahoma State University, Stillwater, OK, USA}
\email{eduardo.teixeira@okstate.edu}


\begin{document}

\begin{abstract}
For all sufficiently large degrees $m$, we prove that the Bombieri--Weyl
norm of a complex $m$-homogeneous polynomial is bounded by its supremum
norm on the unit ball of $\ell_p^n$ whenever $p\ge2m+A$, with constant
one in every dimension and $A$ absolute. At the critical exponent
$p=2m$, our coefficient estimates yield the quantitative bound
$1+O(m^{-1})$ for the optimal comparison constant. Together with the
elementary obstruction to dimension-free boundedness below $2m$,
these results locate the contractivity threshold at $2m+O(1)$. The proof combines compression of permutation orbits with a Wiener coefficient estimate, after separating the coordinate pure powers.
{\color{black}The same estimates yield constant-one symmetric multilinear
Hardy--Littlewood and Bohnenblust--Hille inequalities with the stronger
diagonal norm on the right-hand side; equality occurs precisely for coordinate
pure powers.}
\end{abstract}

\maketitle

\tableofcontents
\medskip

\section{Introduction and main results}

{\color{black}
The Bombieri--Weyl norm is a classical Hilbertian norm on spaces of
homogeneous polynomials.  It is also called the Weyl norm or Bombieri norm in
different parts of the literature.  This terminology reflects two classical
strands of the subject: the invariant Hilbertian geometry associated with the
Weyl norm, and the use of the same coefficient structure in Bombieri-type
inequalities for products of polynomials.  The norm has since become standard
also in the study of polynomial systems and condition numbers; see, for
instance,
\cite{BeauzamyBombieriEnfloMontgomery,ShubSmaleBezoutI}. {\color{black}For further aspects of Bombieri--Weyl geometry, extremal products, differential identities, polynomial-system conditioning, and related approximation questions, see also \cite{BeauzamyExtremalBombieri,BeauzamyDegot,BeltranShubBombieriWeyl,BeltranEtayoMarzoOrtega,DiattaLerarioBombieriWeyl}.}

{\color{black}We write $\TT:=\{z\in\CC:|z|=1\}$ and $\DD:=\{z\in\CC:|z|<1\}$.}

Let
\[
 P:\CC^n\longrightarrow\CC,\qquad
 P(z)=\sum_{|\alpha|=m}b_\alpha z^\alpha
\]
be an $m$-homogeneous polynomial. Here
$\alpha=(\alpha_1,\ldots,\alpha_n)\in\mathbb N_0^n$, where $\mathbb N_0:=\{0,1,2,\ldots\}$, and we use the standard
notation
\[
 |\alpha|=\alpha_1+\cdots+\alpha_n,\qquad
 \alpha!=\alpha_1!\cdots\alpha_n!,\qquad
 z^\alpha=z_1^{\alpha_1}\cdots z_n^{\alpha_n}.
\]
Its Bombieri--Weyl norm is
\[
 \|P\|_{\mathrm{BW}}
 :=\left(\sum_{|\alpha|=m}\frac{\alpha!}{m!}|b_\alpha|^2\right)^{1/2}.
\]

A different line of questions concerns comparisons between coefficient norms
and supremum norms.  In this direction, Galicer, Mansilla, and Muro
\cite{GalicerMansillaMuro} studied the optimal constants relating coefficient
norms of homogeneous polynomials to
\[
 \|P\|_p:=\sup_{\|z\|_p\le1}|P(z)|,\qquad 1\le p\le\infty,
\]
and developed a systematic study of their dimensional asymptotics in a
broad range of parameters.  Their framework also includes the Bombieri
$r$-norm and provides a natural background for the question considered here.
At the Hilbertian value $r=2$, where this norm is precisely the
Bombieri--Weyl norm, we pursue the finer problem of contractivity: for which
$p$ does
\[
 \|P\|_{\mathrm{BW}}\le \|P\|_p
\]
hold for every $m$-homogeneous polynomial $P$, with constant one independently
of the dimension?

There is a sharp obstruction below $2m$.  For
$P_n(z)=z_1^m+\cdots+z_n^m$ one has $\|P_n\|_{\mathrm{BW}}=\sqrt n$, while
{\color{black}$\|P_n\|_p=n^{1-m/p}$ for $p\ge m$, whereas $\|P_n\|_p=1$ for $p<m$.}  Hence no dimension-free estimate
$\|P\|_{\mathrm{BW}}\le C\|P\|_p$ can hold for $p<2m$; compare
\cite[Theorem~2.1]{GalicerMansillaMuro}.

Accordingly, let ${\mathcal P}_m(\CC^n)$ denote the space of complex
$m$-homogeneous polynomials on $\CC^n$, and define
\[
 \mathcal C_{m,p}^{\mathrm{BW}}
 :=\sup_{n\ge1}\ \sup_{0\ne P\in\mathcal P_m(\CC^n)}
 \frac{\|P\|_{\mathrm{BW}}}{\|P\|_p}.
\]
{\color{black}Our main result identifies the contractive regime up to a bounded window
above the critical value $2m$ and shows asymptotic contractivity at the
critical point.}
}

\begin{maintheorem}[Bombieri--Weyl contractivity]
\label{thm:Bombieri-Weyl-threshold}
{\color{black}
There is an absolute constant $C>0$ such that, for every $m\ge1$,
\[
  1\le \mathcal C_{m,2m}^{\mathrm{BW}}\le1+\frac{C}{m}.
\]
Moreover, there are absolute constants $A>0$ and $m_0\in\NN$ such that,
for every $m\ge m_0$ and $p\ge2m+A$,
\[
  \mathcal C_{m,p}^{\mathrm{BW}}=1.
\]
{\color{black}In this contractive regime, equality is rigid: a nonzero polynomial $P$
satisfies $\|P\|_{\mathrm{BW}}=\|P\|_p$ if and only if
\[
 P(z)=c z_j^m
\]
for some $j\in\{1,\ldots,n\}$ and $c\in\CC\setminus\{0\}$.}
}
\end{maintheorem}

{\color{black}The complex field is essential here.  Over the real field the Bombieri--Weyl
comparison fails to be contractive even at $p=\infty$, already in two variables.
The final section gives the elementary obstruction $(x^2-y^2)^{m/2}$, whose
Bombieri--Weyl norm grows like $m^{1/4}$.}

{\color{black}To formulate the coefficient estimate behind this result, recall the
classical Bombieri $q$-norm}
\begin{equation}\label{eq:orbit-norm-definition}
 \Lambda_{m,q}(P)^q
 :=\sum_{|\alpha|=m}
 \left(\frac{m!}{\alpha!}\right)^{1-q}|b_\alpha|^q,
 \qquad q\ge1,
\end{equation}
so that $\Lambda_{m,2}(P)=\|P\|_{\mathrm{BW}}$; see
\cite{BeauzamyBombieriEnfloMontgomery,GalicerMansillaMuro}.  For
$2m\le p\le\infty$, set
\begin{equation}\label{eq:q-isotropic-full}
 q_{m,p}:=\frac{2mp}{mp+p-2m},
 \qquad {\color{black}q_m:=q_{m,\infty}=\frac{2m}{m+1}.}
\end{equation}

\begin{maintheorem}[{\color{black}Dimension-free Bombieri $q$-norm estimate}]
\label{thm:main-A}
There is an absolute constant $C>0$ such that, for every $m,n\in\NN$,
every $p\in[2m,\infty]$, and every complex $m$-homogeneous polynomial
$P:\CC^n\to\CC$,
\begin{equation}\label{eq:main-A-contractivity}
 \Lambda_{m,q_{m,p}}(P)
 \le\left(1+\frac Cm\right)\|P\|_p.
\end{equation}
\end{maintheorem}

The constant $C$ is independent of the degree, the dimension, and $p$.
Thus the comparison factor tends to one uniformly throughout the
admissible range. {\color{black}The next theorem shows that this classical Bombieri $q$-norm comparison has exact constant one once $p$ lies beyond a bounded window above $2m$.}

\medskip
{\color{black}For the exact estimate, set
\[
 s_{m,p}:=\frac{p}{p-m}\quad (2m\le p<\infty),
 \qquad s_{m,\infty}:=1.
\]
This is the exponent governing the diagonal coefficients.  The parameter
$\beta_{m,p}$ below measures the separation between the Hardy--Littlewood
exponent $q_{m,p}$ from \eqref{eq:q-isotropic-full} and $s_{m,p}$.}
\begin{maintheorem}[{\color{black}Bounded-window Bombieri $q$-norm contractivity}]
\label{thm:main-B}
There are absolute constants $A,B>0$ and $m_0\in\NN$ such that, whenever
$m\ge m_0$, $n\in\NN$, and $p\in[2m,\infty]$ satisfies
\begin{equation}\label{eq:beta-criterion-intro}
 \beta_{m,p}:=\frac{q_{m,p}}{s_{m,p}}-1\ge\frac{B}{m},
\end{equation}
every complex $m$-homogeneous polynomial
$P:\CC^n\to\CC$ satisfies
\begin{equation}\label{eq:main-B-exact-subhilbert}
 \Lambda_{m,q_{m,p}}(P)\le\|P\|_p.
\end{equation}
In particular, \eqref{eq:main-B-exact-subhilbert} holds throughout
\[
 p\ge2m+A,
 \qquad m\ge m_0.
\]
{\color{black}Equivalently, the possible failure of exact Bombieri $q$-norm contractivity is confined}
to a strip of bounded width above $2m$.
\end{maintheorem}

For the multilinear statements, let $e_1,\ldots,e_n$ denote the canonical basis of $\CC^n$.
If $T:(\CC^n)^m\to\CC$ is symmetric and
\[
 P_T(z):=T(z,\ldots,z)=\sum_{|\alpha|=m}b_\alpha z^\alpha,
\]
then symmetry gives
\begin{equation}\label{eq:intro-symmetric-coefficients}
 \sum_{i_1,\ldots,i_m=1}^n |T(e_{i_1},\ldots,e_{i_m})|^q
 =\sum_{|\alpha|=m}
 \left(\frac{m!}{\alpha!}\right)^{1-q}|b_\alpha|^q
 =\Lambda_{m,q}(P_T)^q.
\end{equation}
This is the ordered-coefficient form of the Bombieri $q$-norm; it will be
derived again from permutation orbits in \Cref{lem:orbit-decomposition}.

{\color{black}The multilinear consequences take the following classical coefficient form.
The right-hand side is the stronger diagonal norm
$\sup_{\|z\|_p\le1}|T(z,\ldots,z)|$, which is bounded above by the usual
multilinear operator norm.  Contractivity and the classification of equality
are stated separately.}

\begin{corollary}[Symmetric Hardy--Littlewood contractivity]
\label{cor:symmetric-HL}
There is an absolute constant $C>0$ such that, for every $m,n\in\NN$,
every $2m\le p<\infty$, and every symmetric complex $m$-linear form
$T:(\CC^n)^m\to\CC$, {\color{black}with the diagonal norm on the right-hand side,}
\[
 \left(\sum_{i_1,\ldots,i_m=1}^n
 |T(e_{i_1},\ldots,e_{i_m})|^{\frac{2mp}{mp+p-2m}}\right)^{\frac{mp+p-2m}{2mp}}
 \le\left(1+\frac Cm\right)
 \sup_{\|z\|_{\ell_p^n}\le1}|T(z,\ldots,z)|.
\]
Moreover, there are absolute constants $A>0$ and $m_0\in\NN$ such that,
for $m\ge m_0$ and $p\ge2m+A$, the factor $1+C/m$ can be replaced by $1$.
\end{corollary}

\begin{corollary}[Hardy--Littlewood rigidity]
\label{cor:symmetric-HL-rigidity}
{\color{black}Let $A$ and $m_0$ be as in \Cref{cor:symmetric-HL}.} {\color{black}Let $T:(\CC^n)^m\to\CC$ be nonzero.} Let $m\ge m_0$ and $p\ge2m+A$. Equality in the constant-one estimate of
\Cref{cor:symmetric-HL} holds if and only if
\[
 T(z,\ldots,z)=c z_j^m
\]
for some $j\in\{1,\ldots,n\}$ and {\color{black}$c\in\CC\setminus\{0\}$}.
\end{corollary}

\begin{corollary}[Symmetric Bohnenblust--Hille contractivity]
\label{cor:symmetric-BH}
There is an absolute constant $C>0$ such that every symmetric complex
$m$-linear form $T:(\CC^n)^m\to\CC$ satisfies, {\color{black}with the diagonal norm on the right-hand side,}
\[
 \left(\sum_{i_1,\ldots,i_m=1}^n
 |T(e_{i_1},\ldots,e_{i_m})|^{\frac{2m}{m+1}}\right)^{\frac{m+1}{2m}}
 \le\left(1+\frac Cm\right)
 \sup_{\|z\|_{\ell_\infty^n}\le1}|T(z,\ldots,z)|.
\]
{\color{black}Moreover, there is an absolute $m_0\in\NN$ such that, for every $m\ge m_0$, the factor $1+C/m$ can be replaced by $1$.}
\end{corollary}

\begin{corollary}[Bohnenblust--Hille rigidity]
\label{cor:symmetric-BH-rigidity}
{\color{black}Let $m_0$ be as in \Cref{cor:symmetric-BH}.} {\color{black}Let $T:(\CC^n)^m\to\CC$ be nonzero.} For every $m\ge m_0$, equality in the constant-one estimate of
\Cref{cor:symmetric-BH} holds if and only if
\[
 T(z,\ldots,z)=c z_j^m
\]
for some $j\in\{1,\ldots,n\}$ and {\color{black}$c\in\CC\setminus\{0\}$}.
\end{corollary}

{\color{black}The coefficient inequalities used here belong to the classical
Bohnenblust--Hille and Hardy--Littlewood traditions; see
\cite{BH,PracianoPereira,DefantFrerickOrtegaOunaiesSeip} and, for more recent
coefficient-summability developments, \cite{BayartJEMS,PellegrinoTeixeiraQP}. {\color{black}The interpolation framework used in such exponent arguments is standard; see, for instance, \cite{BerghLofstrom}.}}

The isotropic Hardy--Littlewood and Bohnenblust--Hille corollaries follow directly as below.

\begin{proof}[Proof of \Cref{cor:symmetric-HL,cor:symmetric-HL-rigidity,cor:symmetric-BH,cor:symmetric-BH-rigidity}]
Apply \Cref{thm:main-A,thm:main-B} to $P(z)=T(z,\ldots,z)$ and use
\eqref{eq:intro-symmetric-coefficients}. The Bohnenblust--Hille estimate is the
case $p=\infty$.

{\color{black}For the equality statements, $P=0$ is trivial.  If $n=1$, every
$m$-homogeneous polynomial has the form $P(z)=cz^m$, so it is already a
coordinate pure power and the asserted characterization follows directly.}
Otherwise, by homogeneity, assume $\|P\|_p=1$ and put
\[
 h:=\max_j|b_{me_j}|.
\]
The proof of \Cref{thm:main-B} gives strict inequality whenever $h<1$
in the stated contractive range. Equality therefore forces $h=1$.
Choose a coordinate whose pure-power coefficient has modulus one.
For finite $p$, \eqref{eq:weighted-Wiener-slice}, with $h=1$ and
$u\downarrow0$, forces every transverse layer to vanish. For
$p=\infty$, the same conclusion follows from
\Cref{lem:Wiener-slice}. Hence $P(z)=c z_j^m$. Conversely, every
coordinate pure power attains equality.
\end{proof}

\section{\textcolor{black}{Polynomial orbit compression and inverse-multinomial decay}}
\label{sec:orbit-technique}

{\color{black}
Let $P$ be an $m$-homogeneous polynomial and $\check P$ its symmetric
polarization. {\color{black}Thus $\check P$ is the unique symmetric $m$-linear form such that
$P(z)=\check P(z,\ldots,z)$; equivalently, for $|\alpha|=m$,
$ b_\alpha=N_\alpha\,\check P(e_1^{\alpha_1},\ldots,e_n^{\alpha_n})$.}  The ordered coefficients of $\check P$ are constant on
permutation orbits, whereas each coefficient of $P$ collects one entire
orbit.  If the orbit has cardinality $N$ and the corresponding polynomial
coefficient is $b$, its contribution to the ordered $\ell_q$-sum is
\[
 N\left|\frac{b}{N}\right|^q=N^{1-q}|b|^q.
\]
{\color{black}By \eqref{eq:orbit-norm-definition}, this is exactly the weight in the classical Bombieri $q$-norm $\Lambda_{m,q}(P)$.}  For $q>1$, large orbits are
therefore suppressed by the inverse multiplicity factor $N^{1-q}$; the
remaining task is to control the small ones.  The inverse-multinomial estimate,
together with the projections and coloring arguments below, does precisely
that.
}

{\color{black}Throughout this section, write
\[
 P(z)=\sum_{|\alpha|=m}b_\alpha z^\alpha,
\]
and let $\check P$ denote its symmetric polarization.  As in the introduction,
$\alpha=(\alpha_1,\ldots,\alpha_n)$ is a coordinate multi-index.  We use
$\kappa$ only for multiplicity patterns.  We shall also use
\[
 \supp(\alpha):=\{j\in\{1,\ldots,n\}:\alpha_j\ne0\}.
\]}
For $|\alpha|=m$, set
\[
 \mathbf i_\alpha:=
 \bigl(
 \underbrace{1,\ldots,1}_{\alpha_1},
 \underbrace{2,\ldots,2}_{\alpha_2},
 \ldots,
 \underbrace{n,\ldots,n}_{\alpha_n}
 \bigr)\in\{1,\ldots,n\}^m.
\]
The symmetric group $S_m$ acts on $\{1,\ldots,n\}^m$ by permuting
coordinates.  We define the permutation orbit associated with $\alpha$ by
\[
 \mathcal O_\alpha
 :=\{\sigma\mathbf i_\alpha:\sigma\in S_m\}.
\]
Hence
\[
 N_\alpha:=|\mathcal O_\alpha|
 =\frac{m!}{\alpha_1!\cdots\alpha_n!}
 =\frac{m!}{\alpha!}.
\]
For example, if $m=3$ and $\alpha=(2,1,0,\ldots,0)$, then
\[
 \mathbf i_\alpha=(1,1,2)
 \qquad\text{and}\qquad
 \mathcal O_\alpha
 =\{(1,1,2),(1,2,1),(2,1,1)\}.
\]
The permutation orbits satisfy the following identity.
\begin{lemma}\label{lem:orbit-decomposition}
{\color{black}
Let $P(z)=\sum_{|\alpha|=m}b_\alpha z^\alpha$ be $m$-homogeneous, and let
$\check P$ be its symmetric polarization.  Then, for every
$|\alpha|=m$ and every $(i_1,\ldots,i_m)\in\mathcal O_\alpha$,
\[
 \check P(e_{i_1},\ldots,e_{i_m})=\frac{b_\alpha}{N_\alpha}.
\]
Consequently, for every $q>0$,
\begin{equation}\label{eq:intro-orbit-identity}
 \sum_{i_1,\ldots,i_m=1}^{n}
 |\check P(e_{i_1},\ldots,e_{i_m})|^q
 =\sum_{|\alpha|=m}N_\alpha^{\,1-q}|b_\alpha|^q.
\end{equation}
For $q\ge1$, the common value is $\Lambda_{m,q}(P)^q$.
}
\end{lemma}

\begin{proof}
{\color{black}
For a fixed multi-index $\alpha$, symmetry makes
$\check P(e_{i_1},\ldots,e_{i_m})$ constant on $\mathcal O_\alpha$.
Expanding $P(z)=\check P(z,\ldots,z)$, the coefficient of $z^\alpha$ is the
sum of these $N_\alpha$ equal ordered coefficients.  Each is therefore
$b_\alpha/N_\alpha$.  Summing their $q$th powers first over the orbit and
then over the disjoint orbits proves \eqref{eq:intro-orbit-identity}; the
last assertion is \eqref{eq:orbit-norm-definition}.
}
\end{proof}

\par\medskip
{For $|\alpha|=m$, let
\(\pat(\alpha)=(\kappa_1,\ldots,\kappa_r)\)
be the nonzero entries of $\alpha$ arranged in nonincreasing order. Thus
$\kappa_1\ge\cdots\ge\kappa_r\ge1$ and
$\kappa_1+\cdots+\kappa_r=m$. We call $\pat(\alpha)$ the
\emph{multiplicity pattern} of $\alpha$.}
\medskip

For $q>1$, the contribution corresponding to $\alpha$ contains the factor $N_\alpha^{-(q-1)}$.  The required multinomial sums are controlled by the following estimate.

\begin{lemma}\label{lem:inverse-multinomial}
There is an absolute constant $K>1$ such that, for $a\ge1/2$ and
$1\le r\le m$,
\begin{equation}\label{eq:inverse-multinomial}
 \sum_{\substack{{\kappa}_1+\cdots+{\kappa}_r=m\\{\kappa}_j\ge1}}
 \binom m{{\kappa}_1,\ldots,{\kappa}_r}^{-a}
 \le
 \frac{K^{r-1}}
 {\bigl[m(m-1)\cdots(m-r+2)\bigr]^a},
\end{equation}
with the denominator interpreted as $1$ when $r=1$.
\end{lemma}

\begin{proof}
{We use the auxiliary estimate}
\begin{equation}\label{eq:inverse-binomial-uniform}
 \sup_{M\ge2}\sum_{j=1}^{M-1}
 \sqrt{\frac M{\binom Mj}}<\infty.
\end{equation}
By symmetry it suffices, up to a factor two, to take $j\le M/2$.
\textcolor{black}{The summand corresponding to $j=1$ is equal to $1$.}
If $2\le j\le\sqrt M$, then
$\binom Mj\ge(M/j)^j$, so
\[
 \sqrt{\frac M{\binom Mj}}
 \le\sqrt M\left(\frac jM\right)^{j/2}.
\]
The terms $j=2,3$ are $O(M^{-1/2})$, while for $j\ge4$ the right-hand side
is at most $M^{-1/2}M^{-(j-4)/4}$.  {Hence the sum over
$2\le j\le\sqrt M$ is bounded by an absolute constant, uniformly in $M$.}
{\color{black}Indeed, for $M\ge16$,
\[
 \sum_{j=4}^{\lfloor\sqrt M\rfloor}
 M^{-1/2}M^{-(j-4)/4}
 \le M^{-1/2}\sum_{\ell=0}^{\infty}M^{-\ell/4}
 =\frac{M^{-1/2}}{1-M^{-1/4}}
 \le 2M^{-1/2},
\]
and the finitely many values $2\le M<16$ are absorbed into the absolute constant.}
If $\sqrt M<j\le M/2$, the same binomial bound gives
\[
 \sqrt{\frac M{\binom Mj}}\le\sqrt M\,2^{-j/2}.
\]
{Therefore
\[
 \sum_{\sqrt M<j\le M/2}\sqrt{\frac M{\binom Mj}}
 \le \sqrt M\sum_{j>\sqrt M}2^{-j/2},
\]
and the right-hand side is uniformly bounded in $M$.  {\color{black}More explicitly,
\[
 \sqrt M\sum_{j>\sqrt M}2^{-j/2}
 \le \frac{\sqrt M\,2^{-\sqrt M/2}}{1-2^{-1/2}},
\]
which is bounded uniformly for $M\ge2$.}  This proves
\eqref{eq:inverse-binomial-uniform}.}

\textcolor{black}{Let $C_0$ be an absolute upper bound for the supremum in
\eqref{eq:inverse-binomial-uniform}, and fix
$K\ge\max\{2,C_0\}$.}

{We prove \eqref{eq:inverse-multinomial} by induction on the number
$r$ of parts.  {\color{black}If $r=1$, the only admissible tuple is $(m)$, so the left-hand side of
\eqref{eq:inverse-multinomial} equals $\binom m m^{-a}=1$, while the denominator
on the right is the empty product and is also $1$.}  Fix $r\ge2$ and assume that
\eqref{eq:inverse-multinomial} holds for $r-1$ parts, for every admissible
total degree.}
{To apply the induction hypothesis, fix the first part,
${\kappa}_1=j$.  Then
\[
 {\kappa}_2+\cdots+{\kappa}_r=m-j.
\]
Since ${\kappa}_2,\ldots,{\kappa}_r\ge1$, necessarily
$1\le j\le m-r+1$.  The multinomial coefficient factors as
\[
 \binom m{j,{\kappa}_2,\ldots,{\kappa}_r}
 =\binom mj
  \binom{m-j}{{\kappa}_2,\ldots,{\kappa}_r}.
\]
Consequently,
\[
\sum_{\substack{{\kappa}_1+\cdots+{\kappa}_r=m\\{\kappa}_\nu\ge1}}
 \binom m{{\kappa}_1,\ldots,{\kappa}_r}^{-a}
=
 \sum_{j=1}^{m-r+1}\binom mj^{-a}
 \sum_{\substack{{\kappa}_2+\cdots+{\kappa}_r=m-j\\{\kappa}_\nu\ge1}}
 \binom{m-j}{{\kappa}_2,\ldots,{\kappa}_r}^{-a}.
\]
The inner sum has exactly $r-1$ positive parts, so the induction
hypothesis applies to it and gives}
\begin{equation}\label{eq:inverse-multinomial-inductive}
\begin{aligned}
&\sum_{\substack{{\kappa}_1+\cdots+{\kappa}_r=m\\{\kappa}_\nu\ge1}}
 \binom m{{\kappa}_1,\ldots,{\kappa}_r}^{-a} \\
&\qquad\le
 K^{r-2}\sum_{j=1}^{m-r+1}
 \binom mj^{-a}
 \bigl[(m-j)(m-j-1)\cdots(m-j-r+3)\bigr]^{-a}.
\end{aligned}
\end{equation}
{Here and below, an empty product is interpreted as $1$.}
{Put $M:=m-r+2$, so that $m-r+1=M-1$.  For
$1\le j\le M-1$,}
\[
\begin{aligned}
&\bigl[m(m-1)\cdots(m-r+2)\bigr]^a
 \binom mj^{-a}
 \bigl[(m-j)(m-j-1)\cdots(m-j-r+3)\bigr]^{-a}\\
&\qquad=
\left[
\frac{m(m-1)\cdots(m-r+2)}
 {\binom mj\,(m-j)(m-j-1)\cdots(m-j-r+3)}
\right]^a\\
&\qquad=
\left[
\frac{m!/(M-1)!}
 {\bigl(m!/[j!(m-j)!]\bigr)\bigl((m-j)!/(M-j)!\bigr)}
\right]^a\\
&\qquad=
\left[\frac{j!(M-j)!}{(M-1)!}\right]^a
=
\left(\frac M{\binom Mj}\right)^a.
\end{aligned}
\]
{Multiplying \eqref{eq:inverse-multinomial-inductive} by
$[m(m-1)\cdots(m-r+2)]^a$ gives}
\begingroup\small
\[
\bigl[m(m-1)\cdots(m-r+2)\bigr]^a
\sum_{\substack{{\kappa}_1+\cdots+{\kappa}_r=m\\{\kappa}_\nu\ge1}}
 \binom m{{\kappa}_1,\ldots,{\kappa}_r}^{-a}
\le K^{r-2}\sum_{j=1}^{M-1}
\left(\frac M{\binom Mj}\right)^a.
\]
\endgroup
{Since $0<M/\binom Mj\le1$ and $a\ge1/2$,}
\[
\left(\frac M{\binom Mj}\right)^a
\le
\left(\frac M{\binom Mj}\right)^{1/2}.
\]
{Hence, by \eqref{eq:inverse-binomial-uniform},}
\begingroup\small
\[
\bigl[m(m-1)\cdots(m-r+2)\bigr]^a
\sum_{\substack{{\kappa}_1+\cdots+{\kappa}_r=m\\{\kappa}_\nu\ge1}}
 \binom m{{\kappa}_1,\ldots,{\kappa}_r}^{-a}
 \le \textcolor{black}{C_0}K^{r-2}\le K^{r-1}.
\]
\endgroup
{\textcolor{black}{By the choice of $K$, the last inequality is valid.}
Dividing by $[m(m-1)\cdots(m-r+2)]^a$ gives \eqref{eq:inverse-multinomial} and
completes the induction.}
\end{proof}

The coefficient projections needed below can be realized by averaging over colorings.

\begin{lemma}\label{lem:pattern-filter}
Let ${\kappa}=({\kappa}_1,\ldots,{\kappa}_r)$ be positive integers with
${\kappa}_1+\cdots+{\kappa}_r=m$.  Let
\[
 \chi:\{1,\ldots,n\}\longrightarrow\{1,\ldots,r\}
\]
{be any map (equivalently, an $r$-coloring of the coordinate set), and define}
\[
 I_j:=\chi^{-1}(\{j\})\subseteq\{1,\ldots,n\},
 \qquad j=1,\ldots,r.
\]
{Then there exists a linear projection
\[
 \mathcal F_{\chi,{\kappa}}:\mathcal P_m(\CC^n)\longrightarrow\mathcal P_m(\CC^n)
\]
with the following properties:
\begin{enumerate}[label=\textup{(\alph*)},leftmargin=2.2em]
\item \label{item:pattern-action}
\begin{equation}\label{eq:pattern-filter-action}
 \mathcal F_{\chi,\kappa}\!\left(\prod_{j=1}^r z_{i_j}^{\kappa_j}\right)
 =\prod_{j=1}^r z_{i_j}^{\kappa_j},
 \qquad i_j\in I_j\quad(j=1,\ldots,r),
\end{equation}
and
\begin{equation}\label{eq:pattern-filter-zero}
 \mathcal F_{\chi,\kappa}(z^\alpha)=0
\end{equation}
for every other degree-$m$ monomial $z^\alpha$. In particular,
\[
 \mathcal F_{\chi,\kappa}^2=\mathcal F_{\chi,\kappa}.
\]
\item \label{item:pattern-contraction}
For every $1\le p\le\infty$ and every $Q\in\mathcal P_m(\CC^n)$,
\begin{equation}\label{eq:pattern-filter-contractive}
 \|\mathcal F_{\chi,{\kappa}}Q\|_p\le\|Q\|_p.
\end{equation}
\end{enumerate}}
\end{lemma}

\begin{proof}
For $z=(z_i)_{i=1}^n\in\CC^n$, set
\[
 z^{(j)}:=\sum_{i\in I_j}z_i e_i,
 \qquad j=1,\ldots,r.
\]
For $k\ge1$, let
\[
 \mu_k:=\{\zeta\in\TT:\zeta^k=1\}.
\]
{\color{black}For each block $I_j$ introduced in the statement of the lemma, set
\[
 \mu_{\kappa_j}^{\,I_j}:=\{f:I_j\to\mu_{\kappa_j}\},
 \qquad j=1,\ldots,r.
\]
For each $i\in I_j$, let $\nu_i$ denote the uniform probability measure on
$\mu_{\kappa_j}$, and set
\[
 \Omega_{\chi,\kappa}:=\prod_{j=1}^r\mu_{\kappa_j}^{\,I_j},
 \qquad
 \mathbb P_{\chi,\kappa}:=\bigotimes_{j=1}^r\ \bigotimes_{i\in I_j}\nu_i.
\]
Thus $(\Omega_{\chi,\kappa},\mathbb P_{\chi,\kappa})$ is a finite product probability space. We write
\[
 \varepsilon=(\varepsilon_1,\ldots,\varepsilon_n)\in\Omega_{\chi,\kappa},
\]
where $\varepsilon_i\in\mu_{\kappa_j}$ whenever $i\in I_j$, and denote expectation with respect to $\mathbb P_{\chi,\kappa}$ by $\mathbb E_\varepsilon$. In particular, the coordinates $\varepsilon_1,\ldots,\varepsilon_n$ are independent. For $i\in I_j$ and $\ell\ge0$,
\begin{equation}\label{eq:root-unity-moment}
 \mathbb E_\varepsilon[\varepsilon_i^\ell]
 =\frac1{\kappa_j}\sum_{s=0}^{\kappa_j-1}
   \left(e^{2\pi \mathrm i s/\kappa_j}\right)^\ell
 =\frac1{\kappa_j}\sum_{s=0}^{\kappa_j-1}e^{2\pi \mathrm i s\ell/\kappa_j}
 =\begin{cases}
 1,&\kappa_j\mid\ell,\\
 0,&\kappa_j\nmid\ell,
 \end{cases}
 \qquad i\in I_j.
\end{equation}}
Let $d\omega$ denote normalized Haar probability measure on $\TT^r$ and,
for $\varepsilon\in\Omega_{\chi,\kappa}$ and $z\in\CC^n$, define
\[
 \varepsilon z:=(\varepsilon_i z_i)_{i=1}^n\in\CC^n.
\]
The linear operator
\[
 \mathcal F_{\chi,\kappa}:\mathcal P_m(\CC^n)\longrightarrow
 \mathcal P_m(\CC^n)
\]
asserted in the lemma is defined as follows: for each
$Q\in\mathcal P_m(\CC^n)$, its image
$\mathcal F_{\chi,\kappa}(Q)\in\mathcal P_m(\CC^n)$ is the polynomial whose
value at $z\in\CC^n$ is given by
\begin{equation}\label{eq:pattern-projection-definition}
 \bigl(\mathcal F_{\chi,{\kappa}}(Q)\bigr)(z)
 :=\mathbb E_\varepsilon\!\left[
   \int_{\TT^r}
 Q\!\left(\sum_{j=1}^r\omega_j(\varepsilon z)^{(j)}\right)
 \prod_{j=1}^r\overline{\omega_j}^{\,\kappa_j}\,d\omega
 \right].
\end{equation}

\medskip\noindent{\textbf{Proof of \textup{(a)}.}}
We prove that, for every $|\alpha|=m$,
\begin{equation}\label{eq:pattern-filter-monomial-action}
 \mathcal F_{\chi,\kappa}(z^\alpha)
 =\begin{cases}
 z^\alpha,
 &\displaystyle z^\alpha=\prod_{j=1}^r z_{i_j}^{\kappa_j}
   \text{ for some }i_j\in I_j,\\[2mm]
 0,&\text{otherwise}.
 \end{cases}
\end{equation}
Because \eqref{eq:pattern-projection-definition} is linear in $Q$, it is enough
to take $Q(z)=z^\alpha=\prod_{i=1}^n z_i^{\alpha_i}$.  Set
\[
 d_j:=\sum_{i\in I_j}\alpha_i,
 \qquad j=1,\ldots,r.
\]
Then
\[
 z^\alpha\!\left(\sum_{j=1}^r\omega_j(\varepsilon z)^{(j)}\right)
 =\left(\prod_{j=1}^r\omega_j^{d_j}\right)
  \left(\prod_{i=1}^n\varepsilon_i^{\alpha_i}\right)z^\alpha,
\]
so \eqref{eq:pattern-projection-definition} gives
\begin{equation}\label{eq:filter-factorization}
 \mathcal F_{\chi,\kappa}(z^\alpha)
 =z^\alpha\,
 \mathbb E_\varepsilon\!\left(\prod_{i=1}^n\varepsilon_i^{\alpha_i}\right)
 \prod_{j=1}^r
 \int_{\TT}\omega_j^{d_j-\kappa_j}\,d\omega_j.
\end{equation}
For every integer $\ell$,
\[
 \int_{\TT}\omega^\ell\,d\omega
 =\begin{cases}1,&\ell=0,\\0,&\ell\ne0.\end{cases}
\]
Hence the product of torus integrals in \eqref{eq:filter-factorization} vanishes unless
\begin{equation}\label{eq:block-degree-condition}
 d_j=\kappa_j,
 \qquad j=1,\ldots,r.
\end{equation}
{Assume \eqref{eq:block-degree-condition}.  By independence and
\eqref{eq:root-unity-moment},}
\[
 {
 \mathbb E_\varepsilon\!\left[\prod_{i=1}^n\varepsilon_i^{\alpha_i}\right]
 =\prod_{j=1}^r\prod_{i\in I_j}
   \mathbb E_\varepsilon[\varepsilon_i^{\alpha_i}]
 =\prod_{j=1}^r\prod_{i\in I_j}
   \mathbf 1_{\{\kappa_j\mid\alpha_i\}}.}
\]
{Here $\mathbf 1_A$ is $1$ when the statement $A$ holds and
$0$ otherwise.  Equivalently,}
\[
 {
 \mathbb E_\varepsilon\!\left[\prod_{i=1}^n\varepsilon_i^{\alpha_i}\right]
 =
 \begin{cases}
  1,& \kappa_j\mid\alpha_i
       \text{ for every }j=1,\ldots,r\text{ and every }i\in I_j,\\
  0,& \text{otherwise}.
 \end{cases}}
\]
{\color{black}
For each fixed $j$, \eqref{eq:block-degree-condition} gives
\[
 \sum_{i\in I_j}\alpha_i=\kappa_j.
\]
If the expectation in \eqref{eq:filter-factorization} is nonzero, then
\eqref{eq:root-unity-moment} also gives $\kappa_j\mid\alpha_i$ for every
$i\in I_j$.  Hence we may write
\[
 \alpha_i=\kappa_j s_i,\qquad s_i\in\mathbb N_0,\qquad i\in I_j.
\]
Consequently,
\[
 \kappa_j
 =\sum_{i\in I_j}\alpha_i
 =\kappa_j\sum_{i\in I_j}s_i,
\]
and, since $\kappa_j>0$,
\[
 \sum_{i\in I_j}s_i=1.
\]
Thus exactly one index $i_j\in I_j$ satisfies $s_{i_j}=1$, while all the
other $s_i$ vanish.  Equivalently,
$\alpha_{i_j}=\kappa_j$ and $\alpha_i=0$ for $i\in I_j\setminus\{i_j\}$.
Therefore
\[
 z^\alpha=\prod_{j=1}^r z_{i_j}^{\kappa_j},
 \qquad i_j\in I_j,
\]
and every factor in \eqref{eq:filter-factorization} is then equal to $1$.
This proves \eqref{eq:pattern-filter-monomial-action}, hence
\eqref{eq:pattern-filter-action}--\eqref{eq:pattern-filter-zero}.
Since the degree-$m$ monomials form a basis of $\mathcal P_m(\CC^n)$,
\eqref{eq:pattern-filter-monomial-action} also gives directly
\[
 \mathcal F_{\chi,\kappa}^2=\mathcal F_{\chi,\kappa}.
\]
}
{\color{black}Thus the monomial action already proves idempotence on a basis.}

\medskip\noindent{\textbf{Proof of \textup{(b)}.}}
{We keep the notation introduced above:
$\Omega_{\chi,\kappa}$ is the product probability space defined before
\eqref{eq:root-unity-moment}, and $\mathbb E_\varepsilon$ denotes expectation
with respect to its product measure.  For $z\in\CC^n$, $\omega\in\TT^r$, and
$\varepsilon\in\Omega_{\chi,\kappa}$, define}
\[
 {
 u(\omega,\varepsilon,z)
 :=\sum_{j=1}^r\omega_j(\varepsilon z)^{(j)}\in\CC^n.}
\]
{If $i\in I_j$, then
$u(\omega,\varepsilon,z)_i=\omega_j\varepsilon_i z_i$.  Since
$\omega_j,\varepsilon_i\in\TT$,}
\[
 {
 |u(\omega,\varepsilon,z)_i|=|z_i|,
 \qquad i=1,\ldots,n.}
\]
{Consequently, for every $1\le p\le\infty$,}

\[
 {
 \|u(\omega,\varepsilon,z)\|_p=\|z\|_p.}
\]
{Now fix $z\in\CC^n$ with $\|z\|_p\le1$.  Then
$\|u(\omega,\varepsilon,z)\|_p\le1$ for every
$(\omega,\varepsilon)\in\TT^r\times\Omega_{\chi,\kappa}$.  Moreover,}
\[
 {
 \left|\prod_{j=1}^r\overline{\omega_j}^{\,\kappa_j}\right|=1.}
\]
{Therefore \eqref{eq:pattern-projection-definition} gives}
\[
\begin{aligned}
 {
 \bigl|\bigl(\mathcal F_{\chi,\kappa}(Q)\bigr)(z)\bigr|}
 &{\le
 \mathbb E_\varepsilon\!\left[
   \int_{\TT^r}|Q(u(\omega,\varepsilon,z))|\,d\omega\right]}\\
 &{\le
 \mathbb E_\varepsilon\!\left[
   \int_{\TT^r}\|Q\|_p\,d\omega\right]
 =\|Q\|_p.}
\end{aligned}
\]
Taking the supremum over $\|z\|_p\le1$ proves
\eqref{eq:pattern-filter-contractive}.
\end{proof}

{For $1\le r\le m$, let $\mathfrak P_{m,r}$ denote the set of
multiplicity patterns of degree $m$ with $r$ positive parts. For
$\kappa=(\kappa_1,\ldots,\kappa_r)\in\mathfrak P_{m,r}$, set
\[
 N_\kappa:=\binom{m}{\kappa_1,\ldots,\kappa_r}
 =\frac{m!}{\kappa_1!\cdots\kappa_r!}.
\]}

\begin{lemma}\label{lem:coloring-average}
Let $\kappa=(\kappa_1,\ldots,\kappa_r)\in\mathfrak P_{m,r}$ and, for $d\ge1$, define
\[
 {m_d(\kappa):=
 \#\{j\in\{1,\ldots,r\}:\kappa_j=d\}.}
\]
Let
\[
 \chi:\{1,\ldots,n\}\longrightarrow\{1,\ldots,r\}
\]
be random, with the random variables $\chi(1),\ldots,\chi(n)$ independent and
\[
 \mathbb P\bigl(\chi(i)=j\bigr)=\frac1r,
 \qquad i\in\{1,\ldots,n\},\quad j\in\{1,\ldots,r\}.
\]
Then:
\begin{enumerate}
\item {If {\color{black}$\alpha\in\mathbb N_0^n$} satisfies
$|\alpha|=m$, $\pat(\alpha)=\kappa$, and
$|\operatorname{supp}\alpha|=r$, then}
\begin{equation}\label{eq:survival-probability}
 \mathbb P_\chi\!\left(
   \mathcal F_{\chi,\kappa}(z^\alpha)=z^\alpha
 \right)
 ={\color{black}\pi_\kappa}
 :=\frac{\prod_{d\ge1}m_d(\kappa)!}{r^r}.
\end{equation}
\item For every $a>0$,
\begin{equation}\label{eq:unordered-to-ordered}
 {
 \sum_{\kappa\in\mathfrak P_{m,r}}\frac1{{\color{black}\pi_\kappa}}
 \binom{m}{\kappa_1,\ldots,\kappa_r}^{-a}
 =\frac{r^r}{r!}
 \sum_{\substack{\kappa_1+\cdots+\kappa_r=m\\\kappa_j\ge1}}
 \binom{m}{\kappa_1,\ldots,\kappa_r}^{-a}.}
\end{equation}
\end{enumerate}
\end{lemma}

\begin{proof}
\par\medskip
\noindent\textcolor{black}{\textbf{(1)}}\par\smallskip
{Enumerate $\supp(\alpha)=\{i_1,\ldots,i_r\}$. Since
$\pat(\alpha)=\kappa$, the nonzero exponents
$\alpha_{i_1},\ldots,\alpha_{i_r}$ are the parts
$\kappa_1,\ldots,\kappa_r$, up to permutation. Thus the enumeration
of the support is independent of the ordering of the parts of $\kappa$.}
{Fix a map
\[
 \chi:\{1,\ldots,n\}\to\{1,\ldots,r\}
\]
and put
\[
 I_j:=\chi^{-1}(\{j\})\subseteq\{1,\ldots,n\},
 \qquad j=1,\ldots,r.
\]
We determine exactly when this fixed $\chi$ makes $z^\alpha$ survive the
projection.  By \eqref{eq:pattern-filter-monomial-action},}
\begin{equation}\label{eq:coloring-match-condition}
 {
 \mathcal F_{\chi,\kappa}(z^\alpha)=z^\alpha
 \quad\Longleftrightarrow\quad
 \begin{cases}
  \{\chi(i_1),\ldots,\chi(i_r)\}=\{1,\ldots,r\},\\[1mm]
  \alpha_{i_\nu}=\kappa_{\chi(i_\nu)}
       \quad\text{for every }\nu=1,\ldots,r.
 \end{cases}}
\end{equation}
{Since $i_1,\ldots,i_r$ are distinct and there are exactly
$r$ colors, the first condition means that
$\chi(i_1),\ldots,\chi(i_r)$ is a permutation of $1,\ldots,r$.  The second
condition means that the support index $i_\nu$ must receive a color $j$ satisfying
$\kappa_j=\alpha_{i_\nu}$.}

{Fix $d\ge1$.  Since the nonzero exponents of $\alpha$ are precisely the parts of $\kappa$, with multiplicity,}
\[
 {
 \#\{\nu:\alpha_{i_\nu}=d\}
 =\#\{j:\kappa_j=d\}
 =m_d(\kappa).}
\]
{Hence the $m_d(\kappa)$ support indices carrying exponent
$d$ can be matched with the $m_d(\kappa)$ colors $j$ for which $\kappa_j=d$
in exactly $m_d(\kappa)!$ ways.  {\color{black}Equivalently, for each $d$ the admissible choices are the bijections
\[
 \{\nu:\alpha_{i_\nu}=d\}
 \longrightarrow
 \{j:\kappa_j=d\},
\]
and the two sets have the same cardinality $m_d(\kappa)$.}
Multiplying over the distinct values of $d$,
the number of assignments
$(\chi(i_1),\ldots,\chi(i_r))$ for which the monomial survives is}
\[
 {
 \prod_{d\ge1}m_d(\kappa)!.}
\]
{Now return to the random map $\chi$ from the statement of
the lemma.  The variables $\chi(i_1),\ldots,\chi(i_r)$ are independent and
uniform on $\{1,\ldots,r\}$.  Thus there are $r^r$ equally likely assignments,
each with probability $r^{-r}$.  Therefore}
\[
 {
 \mathbb P_\chi\!\left[
   \mathcal F_{\chi,\kappa}(z^\alpha)=z^\alpha
 \right]
 =\frac{\prod_{d\ge1}m_d(\kappa)!}{r^r}
 ={\color{black}\pi_\kappa}.}
\]
{This proves \eqref{eq:survival-probability}.}

\par\medskip
\noindent\textcolor{black}{\textbf{(2)}}\par\smallskip
{If $\pi\in S_r$, then}
\[
 {
 N_{(\kappa_{\pi(1)},\ldots,\kappa_{\pi(r)})}
 =\frac{m!}{\kappa_{\pi(1)}!\cdots\kappa_{\pi(r)}!}
 =N_\kappa,}
\]
{so $N_\kappa$ is unchanged when the parts of $\kappa$ are reordered.}
A fixed unordered pattern $\kappa$ has
\[
 \frac{r!}{\prod_{d\ge1}m_d(\kappa)!}
\]
distinct orderings.  Since $N_\kappa$ is invariant under reordering and, by
\eqref{eq:survival-probability},
\[
 \frac1{{\color{black}\pi_\kappa}}=\frac{r^r}{\prod_{d\ge1}m_d(\kappa)!},
\]
we obtain
\[
 \frac1{{\color{black}\pi_\kappa}}N_\kappa^{-a}
 =\frac{r^r}{r!}
   \frac{r!}{\prod_{d\ge1}m_d(\kappa)!}N_\kappa^{-a}.
\]
The factor
$r!/\prod_{d\ge1}m_d(\kappa)!$ is exactly the number of distinct orderings of
$\kappa$.  Therefore, after summing over
$\kappa\in\mathfrak P_{m,r}$, every ordered positive $r$-tuple with sum $m$
appears exactly once.  This gives \eqref{eq:unordered-to-ordered}.
\end{proof}

\section{{\color{black}Dimension-free control of the Bombieri $q$-norm}}
\label{sec:HL-proof}

{\color{black}Here the Bombieri norm is always taken at the exponent $q_{m,p}=2mp/(mp+p-2m)$ from \eqref{eq:q-isotropic-full}; at $p=\infty$ this means $q_m=2m/(m+1)$.}

{\color{black}
The sole multilinear input is the following Hardy--Littlewood estimate for the
reduced forms associated with individual multiplicity patterns.
}

{\color{black}If $U:(\CC^n)^r\to\CC$ is an $r$-linear form, we write
\[
 a(U):=\bigl(U(e_{i_1},\ldots,e_{i_r})\bigr)_{i_1,\ldots,i_r=1}^n
\]
for its ordered coefficient array.}

\begin{lemma}
\label{lem:reduced-HL}
Let $r\ge1$ and $u_1,\ldots,u_r\in[1,\infty]$ satisfy
$\sigma:=\sum_{j=1}^r1/u_j\le1/2$.  Put
\[
 \rho:=\frac{2r}{r+1-2\sigma}.
\]
Every complex $r$-linear form
$B:\ell_{u_1}^n\times\cdots\times\ell_{u_r}^n\to\CC$ satisfies
\begin{equation}\label{eq:reduced-HL}
 \|a(B)\|_\rho\le2^{(r-1)/2}\|B\|.
\end{equation}
\textcolor{black}{Here $\|B\|$ denotes the usual operator norm on the product
of the indicated unit balls.}
\end{lemma}

\begin{proof}
{Apply {\color{black}\hypersetup{citecolor=black}\cite[Theorem~2.2, inequality~(2.2)]%
{AlbuquerqueBayartPellegrinoSeoaneHL}} with $Y=\CC$, $v=\mathrm{id}_{\CC}$,
cotype $2$, and summability parameter $1$.  In the notation of that theorem,
$\ell=(1-\sigma)^{-1}\in[1,2]$.  Its admissibility conditions for
$(t_1,\ldots,t_r)$ are}
\[
 {\ell\le t_j\le2\quad(j=1,\ldots,r),
 \qquad
 \sum_{j=1}^r\frac1{t_j}\mathrel{\textcolor{black}{\le}}\frac1{\ell}+\frac{r-1}{2}.}
\]
{For the choice $t_1=\cdots=t_r=\rho$, these conditions reduce to}
\[
 {\frac r\rho=\frac1{\ell}+\frac{r-1}{2},
 \qquad \ell\le\rho\le2,}
\]
{so the choice is admissible.}
{For the scalar field, the cotype-$2$ constant of $\CC$ and the
$(1,1)$-summing norm of $\mathrm{id}_{\CC}$ are both equal to $1$.  Hence
the constant supplied by that theorem is $(\sqrt2)^{r-1}$, which gives
\eqref{eq:reduced-HL}.}
\end{proof}

\medskip
{Fix $2m\le p\le\infty$ and write}
\[
 {q:=q_{m,p},\qquad \eta:=q-1.}
\]
{\color{black}Recall that, for
$\kappa=(\kappa_1,\ldots,\kappa_r)\in\mathfrak P_{m,r}$,
\[
 N_\kappa=\frac{m!}{\kappa_1!\cdots\kappa_r!},
 \qquad
 \pi_\kappa=\frac{\prod_{d\ge1}m_d(\kappa)!}{r^r},
\]
where $m_d(\kappa)$ is the number of parts of $\kappa$ equal to $d$.
Thus $N_\kappa$ is the orbit size attached to the pattern and $\pi_\kappa$
is its survival probability in the coloring argument.}
{Then}
\begin{equation}\label{eq:eta-range}
 {\frac{m-1}{m+1}\le \eta\le1.}
\end{equation}

{Let $Q\in\mathcal P_m(\CC^n)$, write}
\[
 {Q(z)=\sum_{|\alpha|=m}b_\alpha z^\alpha,}
\]
{and fix \textcolor{black}{$1\le r\le m$ and $\kappa\in\mathfrak P_{m,r}$}.  Define}
\[
 A_\kappa:\mathcal P_m(\CC^n)\longrightarrow[0,\infty)
\]
by
\begin{equation}\label{eq:pattern-contribution-HL}
 A_{\kappa}(Q)
 :=\left(\sum_{\pat(\alpha)={\kappa}}N_{\kappa}^{1-q}|b_\alpha|^q\right)^{1/q}.
\end{equation}
{\color{black}
The quantity $A_\kappa$ satisfies the identity
\[
 \sum_{r=1}^m\ \sum_{\kappa\in\mathfrak P_{m,r}}A_\kappa(Q)^q
 =\sum_{|\alpha|=m}N_\alpha^{1-q}|b_\alpha|^q
 =\Lambda_{m,q}(Q)^q,
\]
so that $A_\kappa(Q)^q$ is the contribution of the multiplicity
pattern $\kappa$ to the {\color{black}Bombieri $q$-norm}.}

\begin{lemma}\label{lem:one-pattern-HL}
\textcolor{black}{Let $m,n\in\NN$, let $p\in[2m,\infty]$, set
$q:=q_{m,p}$ and $\eta:=q-1$, and let
$Q(z)=\sum_{|\alpha|=m}b_\alpha z^\alpha\in\mathcal P_m(\CC^n)$.
Then, for every $1\le r\le m$ and every
$\kappa\in\mathfrak P_{m,r}$,}
\begin{equation}\label{eq:one-pattern-HL}
 A_{\kappa}(Q)^q
 \le
 2^{q(r-1)/2}r^{qm/p}\,\pi_{\kappa}^{-1}N_{\kappa}^{-\eta}
 \|Q\|_p^q,
\end{equation}
with the convention $m/p=0$ at $p=\infty$.
\end{lemma}

\begin{proof}
{Fix
\[
 \chi:\{1,\ldots,n\}\longrightarrow\{1,\ldots,r\},
 \qquad I_j:=\chi^{-1}(\{j\}),\quad j=1,\ldots,r.
\]
Recall from \Cref{lem:pattern-filter} that
\[
 \mathcal F_{\chi,\kappa}:\mathcal P_m(\CC^n)\longrightarrow
 \mathcal P_m(\CC^n)
\]
is the contractive coefficient projection which keeps exactly the monomials
of the form
\[
 \prod_{j=1}^r z_{i_j}^{\kappa_j},
 \qquad i_j\in I_j\quad(j=1,\ldots,r),
\]
and annihilates all other degree-$m$ monomials.  Set}
\[
 {
 P_{\chi,\kappa}:=\mathcal F_{\chi,\kappa}(Q).}
\]
{More precisely, \Cref{lem:pattern-filter} defines this projection by
\[
 \bigl(\mathcal F_{\chi,\kappa}(Q)\bigr)(z)
 =
 \mathbb E_\varepsilon\!\left[
   \int_{\TT^r}
   Q\!\left(\sum_{j=1}^r\omega_j(\varepsilon z)^{(j)}\right)
   \prod_{j=1}^r\overline{\omega_j}^{\,\kappa_j}\,d\omega
 \right],
\]
where the notation $\varepsilon$, $(\varepsilon z)^{(j)}$ and the product
probability space are those introduced in that lemma.  By
\eqref{eq:pattern-filter-monomial-action},}
\[
 P_{\chi,{\kappa}}(z)
 =\sum_{\substack{i_j\in I_j\\ j=1,\ldots,r}}
 c_{i_1,\ldots,i_r}\prod_{j=1}^r z_{i_j}^{{\kappa}_j}.
\]
{Define}
\[
 {B_{\chi,{\kappa}}:
 \ell_{p/{\kappa}_1}^n\times\cdots\times
 \ell_{p/{\kappa}_r}^n\longrightarrow\CC}
\]
{by}
\begin{equation}\label{eq:reduced-form-definition}
 {B_{\chi,{\kappa}}(y^{(1)},\ldots,y^{(r)})
 :=\sum_{\substack{i_j\in I_j\\ j=1,\ldots,r}}
 c_{i_1,\ldots,i_r}\prod_{j=1}^r y_{i_j}^{(j)}.}
\end{equation}
{For every
$(i_1,\ldots,i_r)\in I_1\times\cdots\times I_r$,
\[
 B_{\chi,\kappa}(e_{i_1},\ldots,e_{i_r})=c_{i_1,\ldots,i_r},
 \qquad
 \|a(B_{\chi,\kappa})\|_q^q
 =\sum_{\substack{i_j\in I_j\\ j=1,\ldots,r}}
 |c_{i_1,\ldots,i_r}|^q.
\]}
The reciprocal sum of its domain exponents is
\[
 \sum_{j=1}^r\frac{{\kappa}_j}{p}=\frac mp\le\frac12.
\]
Thus \Cref{lem:reduced-HL} applies at the exponent
\[
 \rho_r=\frac{2r}{r+1-2m/p}.
\]
{Using
\[
 \rho_r=\frac{2r}{r+1-2m/p}
 \qquad\text{and}\qquad
 q=\frac{2mp}{mp+p-2m},
\]
we obtain}
\begin{equation}\label{eq:reduced-exponent-comparison}
 \frac1{\rho_r}-\frac1q
 =\frac{m-r}{2mr}\left(1-\frac{2m}{p}\right)\ge0.
\end{equation}
Hence $\rho_r\le q$.  Since finite-dimensional $\ell_t$ norms decrease as the exponent increases,
{
\[
 \|a(B_{\chi,{\kappa}})\|_q
 \le \|a(B_{\chi,{\kappa}})\|_{\rho_r}.
\]
Thus the Hardy--Littlewood estimate at $\rho_r$ gives the required estimate at $q$.}

{ To estimate $\|B_{\chi,{\kappa}}\|$, fix
\[
 y^{(j)}\in\ell_{p/{\kappa}_j}^n,
 \qquad \|y^{(j)}\|_{p/{\kappa}_j}\le1,
 \qquad j=1,\ldots,r.
\]
For each $j$, {let
$\arg:\CC\setminus\{0\}\to(-\pi,\pi]$ denote the principal argument, so that
$u=|u|e^{\mathrm i\arg u}$ for every $u\ne0$.  Define}
$z^{(j)}=(z_i^{(j)})_{i=1}^n$ by
\[
 z_i^{(j)}:=
 \begin{cases}
 |y_i^{(j)}|^{1/{\kappa}_j}
 \exp\!\left(\dfrac{\mathrm i\,\arg y_i^{(j)}}{{\kappa}_j}\right),
   & i\in I_j,\ y_i^{(j)}\ne0,\\[6pt]
 0,&\text{otherwise}.
 \end{cases}
\]
Then
\[
 (z_i^{(j)})^{{\kappa}_j}
 =\begin{cases}
 y_i^{(j)},&i\in I_j,\\
 0,&i\notin I_j,
 \end{cases}
\]
and, when $p<\infty$,
\[
 \|z^{(j)}\|_p^p
 =\sum_{i\in I_j}|y_i^{(j)}|^{p/{\kappa}_j}
 \le\sum_{i=1}^n|y_i^{(j)}|^{p/{\kappa}_j}
 =\|y^{(j)}\|_{p/{\kappa}_j}^{p/{\kappa}_j}
 \le1.
\]
For $p=\infty$ the same definition gives
$\|z^{(j)}\|_\infty\le1$ directly.  Put
\[
 c_j:=\left(\frac{{\kappa}_j}{m}\right)^{1/p},
 \qquad
 w:=\sum_{j=1}^r c_jz^{(j)},
\]
with $c_j:=1$ when $p=\infty$.  Since the supports of the vectors
$z^{(j)}$ are contained in the pairwise disjoint sets $I_j$, for
$p<\infty$ we have
\[
 {
 \|w\|_p^p
 =\sum_{j=1}^r c_j^p\|z^{(j)}\|_p^p
 \le\sum_{j=1}^r c_j^p
 =\sum_{j=1}^r\frac{\kappa_j}{m}
 =\frac{\kappa_1+\cdots+\kappa_r}{m}=1.}
\]
For $p=\infty$, disjointness gives
$\|w\|_\infty=\max_j\|z^{(j)}\|_\infty\le1$.

{Since $i_j\in I_j$, the $i_j$th coordinate of $w$ is
$c_jz_{i_j}^{(j)}$. Hence}
\[
 \begin{aligned}
 P_{\chi,{\kappa}}(w)
 &=\sum_{\substack{i_j\in I_j\\ j=1,\ldots,r}}
 c_{i_1,\ldots,i_r}
 \prod_{j=1}^r(c_jz_{i_j}^{(j)})^{{\kappa}_j}\\
 &=\left(\prod_{j=1}^r c_j^{{\kappa}_j}\right)
 \sum_{\substack{i_j\in I_j\\ j=1,\ldots,r}}
 c_{i_1,\ldots,i_r}
 \prod_{j=1}^r(z_{i_j}^{(j)})^{{\kappa}_j}\\
 &=\left(\prod_{j=1}^r c_j^{{\kappa}_j}\right)
 B_{\chi,{\kappa}}(y^{(1)},\ldots,y^{(r)}).
 \end{aligned}
\]
Therefore
\[
 B_{\chi,{\kappa}}(y^{(1)},\ldots,y^{(r)})
 =\left(\prod_{j=1}^r c_j^{-{\kappa}_j}\right)
 P_{\chi,{\kappa}}(w).
\]
Since $\|w\|_p\le1$ and, by \Cref{lem:pattern-filter},
$\|P_{\chi,{\kappa}}\|_p\le\|Q\|_p$, it follows that
\[
 \begin{aligned}
 |B_{\chi,{\kappa}}(y^{(1)},\ldots,y^{(r)})|
 &\le\left(\prod_{j=1}^r c_j^{-{\kappa}_j}\right)
       \|P_{\chi,{\kappa}}\|_p\\
 &\le\left(\prod_{j=1}^r
 \left(\frac m{{\kappa}_j}\right)^{{\kappa}_j/p}\right)
 \|Q\|_p.
 \end{aligned}
\]
{Set
\[
 G_p(\kappa):=\prod_{j=1}^r
 \left(\frac m{\kappa_j}\right)^{\kappa_j/p}.
\]}
{Taking the supremum over
\[
 (y^{(1)},\ldots,y^{(r)})\in
 \prod_{j=1}^r
 \bigl\{y\in\ell_{p/\kappa_j}^n:\|y\|_{p/\kappa_j}\le1\bigr\}
\]
gives}
\begin{equation}\label{eq:reduced-form-norm}
 {
 \|B_{\chi,\kappa}\|
 =
 \sup_{\substack{y^{(j)}\in\ell_{p/\kappa_j}^n\\
 \|y^{(j)}\|_{p/\kappa_j}\le1,\ 1\le j\le r}}
 |B_{\chi,\kappa}(y^{(1)},\ldots,y^{(r)})|
 \le G_p(\kappa)\|Q\|_p.}
\end{equation}}
{
Put $x_j:={\kappa}_j/m$.  Then $x_j>0$ and
$\sum_{j=1}^r x_j=1$.  From the definition of $G_p({\kappa})$,
\[
 \begin{aligned}
 \log G_p({\kappa})
 &=\frac1p\sum_{j=1}^r {\kappa}_j
      \log\!\left(\frac m{{\kappa}_j}\right)\\
 &=\frac mp\sum_{j=1}^r x_j\log\!\left(\frac1{x_j}\right)\\
 &\le \frac mp\log r.
 \end{aligned}
\]
Exponentiating gives
\begin{equation}\label{eq:entropy-factor}
 G_p({\kappa})\le r^{m/p}.
\end{equation}
By \Cref{lem:reduced-HL}, specifically \eqref{eq:reduced-HL}, together with the comparison $\rho_r\le q$, \eqref{eq:reduced-form-norm}, and \eqref{eq:entropy-factor},
\[
 {\begin{aligned}
 \|a(B_{\chi,{\kappa}})\|_q
 &\mathrel{\overset{\rho_r\le q}{\le}} \|a(B_{\chi,{\kappa}})\|_{\rho_r}\\
 &\mathrel{\overset{\eqref{eq:reduced-HL}}{\le}} 2^{(r-1)/2}\|B_{\chi,{\kappa}}\|\\
 &\mathrel{\overset{\eqref{eq:reduced-form-norm}}{\le}} 2^{(r-1)/2}G_p({\kappa})\|Q\|_p\\
 &\mathrel{\overset{\eqref{eq:entropy-factor}}{\le}} 2^{(r-1)/2}r^{m/p}\|Q\|_p.
 \end{aligned}}
\]
Equivalently,
\begin{equation}\label{eq:colored-array-bound-explicit}
 \sum_{\substack{\pat(\alpha)={\kappa}\\
                  z^\alpha\ \text{is left unchanged by }\mathcal F_{\chi,{\kappa}}}}
 |b_\alpha|^q
 \le
 2^{q(r-1)/2}r^{qm/p}\|Q\|_p^q.
\end{equation}

{
Let $\chi$ be uniformly distributed on
$\{1,\ldots,r\}^{\{1,\ldots,n\}}$, and, for
$\pat(\alpha)=\kappa$, set
\[
 \mathbf 1_\chi(\alpha):=
 \begin{cases}
 1,&\mathcal F_{\chi,\kappa}(z^\alpha)=z^\alpha,\\
 0,&\mathcal F_{\chi,\kappa}(z^\alpha)=0.
 \end{cases}
\]
By \Cref{lem:coloring-average},
\[
 \mathbb E_\chi[\mathbf 1_\chi(\alpha)]={\color{black}\pi_\kappa}.
\]
}
Hence, using the definition of $A_{\kappa}(Q)$ and the fact that
$\eta=q-1$,
\[
 \begin{aligned}
 \pi_{\kappa}A_{\kappa}(Q)^q
 &=\pi_{\kappa}N_{\kappa}^{1-q}
   \sum_{\pat(\alpha)={\kappa}}|b_\alpha|^q\\
 &=N_{\kappa}^{-\eta}
   \sum_{\pat(\alpha)={\kappa}}
   \mathbb E_\chi[\mathbf 1_\chi(\alpha)]|b_\alpha|^q\\
 &=N_{\kappa}^{-\eta}\,
   \mathbb E_\chi
   \sum_{\pat(\alpha)={\kappa}}
   \mathbf 1_\chi(\alpha)|b_\alpha|^q\\
 &\mathrel{\overset{\eqref{eq:colored-array-bound-explicit}}{\le}}
 N_{\kappa}^{-\eta}
 2^{q(r-1)/2}r^{qm/p}\|Q\|_p^q,
 \end{aligned}
\]
  Dividing by $\pi_{\kappa}>0$ yields
\[
 A_{\kappa}(Q)^q
 \le
 2^{q(r-1)/2}r^{qm/p}\pi_{\kappa}^{-1}N_{\kappa}^{-\eta}
 \|Q\|_p^q,
\]
which is \eqref{eq:one-pattern-HL}.}

\end{proof}

\medskip
Assume $m/p\le1/2$ and consider first the patterns of length $r\ge2$.

{
Summing \Cref{lem:one-pattern-HL} over
$\bigcup_{r=2}^m\mathfrak P_{m,r}$ gives the following estimate.
}

\begin{proposition}
\label{prop:HL-master}
\textcolor{black}{There is an absolute constant $C>0$ such that, for every
$m,n\in\NN$ with $m\ge3$, every $p\in[2m,\infty]$, and every
$Q(z)=\sum_{|\alpha|=m}b_\alpha z^\alpha\in\mathcal P_m(\CC^n)$,
with $q:=q_{m,p}$ and $\eta:=q-1$,}
\begin{equation}\label{eq:HL-master}
 \sum_{r=2}^m\ {\sum_{{\kappa}\in\mathfrak P_{m,r}}} A_{\kappa}(Q)^q
 \le C m^{-\eta}\|Q\|_p^q.
\end{equation}
\end{proposition}

\begin{proof}
{
Fix $r\in\{2,\ldots,m\}$. By \eqref{eq:survival-probability},
\eqref{eq:unordered-to-ordered}, and \Cref{lem:inverse-multinomial},
\[
\begin{aligned}
 \sum_{\kappa\in\mathfrak P_{m,r}} A_\kappa(Q)^q
 &\mathrel{\overset{\Cref{lem:one-pattern-HL}}{\le}}
 2^{q(r-1)/2}r^{qm/p}\|Q\|_p^q
 \sum_{\kappa\in\mathfrak P_{m,r}}
 \frac1{{\color{black}\pi_\kappa}}N_\kappa^{-\eta}\\
 &\mathrel{\overset{\eqref{eq:unordered-to-ordered}}{=}}
 2^{q(r-1)/2}r^{qm/p}\|Q\|_p^q
 \frac{r^r}{r!}
 \sum_{\substack{\kappa_1+\cdots+\kappa_r=m\\\kappa_j\ge1}}
 \binom{m}{\kappa_1,\ldots,\kappa_r}^{-\eta}\\
 &\mathrel{\overset{\Cref{lem:inverse-multinomial}}{\le}}
 2^{q(r-1)/2}r^{qm/p}\frac{r^r}{r!}
 \frac{K^{r-1}}
 {\bigl[m(m-1)\cdots(m-r+2)\bigr]^\eta}\|Q\|_p^q.
\end{aligned}
\]

Since $q\le2$, $qm/p\le1$, and $r^r/r!\le\e^r$, for $r\ge2$,
\[
 2^{q(r-1)/2}r^{qm/p}\frac{r^r}{r!}K^{r-1}
 \le
 2^{r-1}r\,\e^r K^{r-1}
 \le r{\color{black}D_0^r},
 \qquad {\color{black}D_0:=2\e\max\{1,K\}}.
\]
Consequently,
\begin{equation}\label{eq:length-r-bound}
 \sum_{\kappa\in\mathfrak P_{m,r}} A_\kappa(Q)^q
 \le
 \frac{r{\color{black}D_0^r}}{\bigl[m(m-1)\cdots(m-r+2)\bigr]^\eta}\|Q\|_p^q.
\end{equation}

Put
\[
 {\color{black}\gamma_r:=\frac{rD_0^r}{[m(m-1)\cdots(m-r+2)]^\eta}},
 \qquad 2\le r\le m.
\]
For $2\le r\le\lfloor m/2\rfloor+1$,
\[
 {\color{black}\frac{\gamma_{r+1}}{\gamma_r}}
 =\frac{r+1}{r}{\color{black}D_0}(m-r+1)^{-\eta}
 \le2{\color{black}D_0}(2/m)^\eta
 \le2{\color{black}D_0}(2/m)^{1/2}.
\]
Choose $m_0$ so that
\[
 2{\color{black}D_0}(2/m)^{1/2}\le\frac12,\qquad m\ge m_0.
\]
Assume first that $m\ge m_0$. Then
\[
 {\color{black}\sum_{r=2}^{\lfloor m/2\rfloor+1}\gamma_r}
 \le {\color{black}\gamma_2}\sum_{j=0}^{\infty}2^{-j}
 =2{\color{black}\gamma_2}=4{\color{black}D_0^2}m^{-\eta}.
\]
For $\lfloor m/2\rfloor+2\le r\le m$,
\[
 m(m-1)\cdots(m-r+2)
 \ge\prod_{j=0}^{\lfloor m/2\rfloor-1}(m-j)
 \ge(m/2)^{\lfloor m/2\rfloor},
\]
and hence
\[
\begin{aligned}
 {\color{black}\sum_{r=\lfloor m/2\rfloor+2}^{m}\gamma_r}
 &\le m^2{\color{black}D_0^m}(m/2)^{-\eta\lfloor m/2\rfloor}\\
 &\le m^2{\color{black}D_0^m}(m/2)^{-\lfloor m/2\rfloor/2}.
\end{aligned}
\]
{\color{black}Here the negative term has order $m\log m$, whereas the two positive terms have order at most $m$: indeed,
\[
 \frac{\lfloor m/2\rfloor}{2}\log(m/2)
 =\left(\frac14+o(1)\right)m\log m,
 \qquad
 2\log m+m\log D_0=O(m).
\]}
For some absolute $c>0$ and all sufficiently large $m$,
\[
\begin{aligned}
 \log\!\left(m^2{\color{black}D_0^m}(m/2)^{-\lfloor m/2\rfloor/2}\right)
 &=2\log m+m\log {\color{black}D_0}-\frac{\lfloor m/2\rfloor}{2}\log(m/2)\\
 &\le-cm\log m,
\end{aligned}
\]
so, after increasing $m_0$ if necessary,
\[
 {\color{black}\sum_{r=\lfloor m/2\rfloor+2}^{m}\gamma_r}
 \le e^{-cm\log m}\le m^{-1}\le m^{-\eta}.
\]
Therefore
\[
 {\color{black}\sum_{r=2}^{m}\gamma_r}
 \le {\color{black}(4D_0^2+1)}m^{-\eta},
 \qquad m\ge m_0.
\]
Combining this with \eqref{eq:length-r-bound} proves
\eqref{eq:HL-master} for every $m\ge m_0$.
}
{For $3\le m<m_0$, \textcolor{black}{fix} one such $m$ and
$r\in\{2,\ldots,m\}$.  From \eqref{eq:length-r-bound}, it is enough to
compare
\[
 \frac{r{\color{black}D_0^r}}{[m(m-1)\cdots(m-r+2)]^\eta}
\]
with $m^{-\eta}$.  Dividing the first quantity by $m^{-\eta}$ gives
\[
 r{\color{black}D_0^r}
 \left(
   \frac{m}{m(m-1)\cdots(m-r+2)}
 \right)^\eta.
\]
By \eqref{eq:eta-range},
\[
 \frac{m-1}{m+1}\le \eta\le1.
\]
For fixed $m$ and $r$, the function
\[
 \eta\longmapsto
 r{\color{black}D_0^r}
 \left(
   \frac{m}{m(m-1)\cdots(m-r+2)}
 \right)^\eta
\]
is continuous on the compact interval
$[(m-1)/(m+1),1]$, and therefore has a finite maximum, say $C_{m,r}$.
Since there are only finitely many pairs
\[
 (m,r),\qquad 3\le m<m_0,\quad 2\le r\le m,
\]
the number
\[
 C_{\rm fin}:=
 \max_{\substack{3\le m<m_0\\2\le r\le m}} C_{m,r}
\]
is finite.  Thus \eqref{eq:length-r-bound} gives, for every one of these
remaining degrees,
\[
 \sum_{\kappa\in\mathfrak P_{m,r}}A_\kappa(Q)^q
 \le C_{\rm fin}m^{-\eta}\|Q\|_p^q.
\]
Summing over $r=2,\ldots,m$ gives
\[
\begin{aligned}
 \sum_{r=2}^{m}\sum_{\kappa\in\mathfrak P_{m,r}}A_\kappa(Q)^q
 &\le (m-1)C_{\rm fin}m^{-\eta}\|Q\|_p^q\\
 &\le (m_0-1)C_{\rm fin}m^{-\eta}\|Q\|_p^q,
\end{aligned}
\]
because $m<m_0$.  Enlarging the absolute constant by the fixed factor
$(m_0-1)C_{\rm fin}$ proves \eqref{eq:HL-master} for every $m\ge3$.}
\end{proof}

\medskip
{\color{black}\hypersetup{linkcolor=black}
\noindent\textbf{Completion of the polynomial estimate.}
Assume first that $m\ge3$, and let
\[
 P(z)=\sum_{|\alpha|=m}b_\alpha z^\alpha\in\mathcal P_m(\CC^n),
 \qquad
 {\color{black}P_{\mathrm{diag}}}(z):=\sum_{j=1}^n b_{me_j}z_j^m,
 \qquad
 R:=P-{\color{black}P_{\mathrm{diag}}}.
\]
The polynomial $R$ contains exactly the terms whose multiplicity patterns
have length at least two.  By \Cref{prop:HL-master}, those terms are already
controlled.  It remains to treat the pure powers.

Put
\[
 \Omega_m:=\{\omega\in\TT:\omega^m=1\}.
\]
Define the linear operator
\[
 \mathcal D:\mathcal P_m(\CC^n)\longrightarrow\mathcal P_m(\CC^n)
\]
by
\[
 \bigl(\mathcal DQ\bigr)(z):=
 \frac{1}{m^n}
 \sum_{(\omega_1,\ldots,\omega_n)\in\Omega_m^n}
 Q(\omega_1z_1,\ldots,\omega_nz_n).
\]
For every $Q(z)=\sum_{|\alpha|=m}c_\alpha z^\alpha$,
\[
 \bigl(\mathcal DQ\bigr)(z)
 =\sum_{|\alpha|=m}c_\alpha z^\alpha
 \prod_{k=1}^n\left(\frac1m\sum_{\omega\in\Omega_m}
 \omega^{\alpha_k}\right).
\]
Since
\[
 \frac1m\sum_{\omega\in\Omega_m}\omega^d
 =\begin{cases}
 1,&d=0\text{ or }d=m,\\
 0,&1\le d\le m-1,
 \end{cases}
\]
and $|\alpha|=m$, only the multi-indices $me_j$ survive.  Hence
\begin{equation}\label{eq:diagonal-projection}
 \mathcal DP={\color{black}P_{\mathrm{diag}}}.
\end{equation}
Coordinatewise multiplication by elements of $\TT$ preserves the
$\ell_p$-ball, so the averaging operator is contractive:
\[
 \|{\color{black}P_{\mathrm{diag}}}\|_p\le\|P\|_p.
\]
Moreover, for every $p\in[2m,\infty]$, the map
$z\mapsto(z_j^m)_j$ sends the unit ball of $\ell_p^n$ onto that of
$\ell_{p/m}^n$.  {\color{black}For $p<\infty$ this follows from the exact identity
\[
 \|(z_j^m)_j\|_{p/m}^{p/m}
 =\sum_j|z_j^m|^{p/m}
 =\sum_j|z_j|^p
 =\|z\|_p^p,
\]
while for $p=\infty$ one has $\|(z_j^m)_j\|_\infty=\|z\|_\infty^m$.}
Surjectivity follows by taking coordinatewise complex
$m$th roots.  Duality therefore gives
\[
 \|{\color{black}P_{\mathrm{diag}}}\|_p=\|(b_{me_j})_j\|_{p/(p-m)},
\]
with the limiting convention $p/(p-m):=1$ when $p=\infty$.
For $p<\infty$,
\[
 \frac{1}{p/(p-m)}-\frac1{q_{m,p}}
 =\frac{p-m}{p}-\frac{mp+p-2m}{2mp}
 =\frac{(m-1)(p-2m)}{2mp}\ge0.
\]
Thus $q_{m,p}\ge p/(p-m)$.  {\color{black}At $p=\infty$ our conventions give
$p/(p-m)=1$ and $q_{m,\infty}=2m/(m+1)\ge1$, so the same comparison holds.}
Since finite-dimensional $\ell_t$-norms decrease with $t$, and the orbit
weights of the pure powers are one,
\begin{equation}\label{eq:principal-HL}
 \Lambda_{m,q}({\color{black}P_{\mathrm{diag}}})
 =\|(b_{me_j})_j\|_q
 \le\|(b_{me_j})_j\|_{p/(p-m)}
 =\|{\color{black}P_{\mathrm{diag}}}\|_p
 \le\|P\|_p.
\end{equation}

Applying \Cref{prop:HL-master} first to $P$ and then to $R$ gives the
polynomial off-diagonal estimate
\begin{equation}\label{eq:off-diagonal-isotropic}
 \Lambda_{m,q}(R)^q
 \le C m^{-\eta}\min\{\|P\|_p^q,\|R\|_p^q\}.
\end{equation}
The coefficient supports of ${\color{black}P_{\mathrm{diag}}}$ and $R$ are disjoint.  Therefore
\[
\begin{aligned}
 \Lambda_{m,q}(P)^q
 &=\Lambda_{m,q}({\color{black}P_{\mathrm{diag}}})^q+\Lambda_{m,q}(R)^q\\
 &\le \|P\|_p^q
   +C m^{-\eta}\min\{\|P\|_p^q,\|R\|_p^q\}\\
 &\le(1+C m^{-\eta})\|P\|_p^q.
\end{aligned}
\]
By \eqref{eq:eta-range}, $\eta\ge(m-1)/(m+1)$, and hence
\[
 m^{-\eta}
 \le m^{-(m-1)/(m+1)}
 =\frac1m\exp\!\left(\frac{2\log m}{m+1}\right)
 \le\frac{\textcolor{black}{C_\eta}}{m},
\]
where ${\color{black}C_\eta}:=\sup_{m\ge3}\exp(2\log m/(m+1))<\infty$.  After absorbing
${\color{black}C_\eta}$ into $C$ and taking $q$th roots,
\[
 \Lambda_{m,q}(P)
 \le\left(1+\frac Cm\right)^{1/q}\|P\|_p
 \le\left(1+\frac Cm\right)\|P\|_p.
\]
This proves \Cref{thm:main-A} for $m\ge3$.

{\color{black}For $m=1$, one has $q_{1,p}=p'$ (with the usual convention at $p=\infty$), and by duality
$\Lambda_{1,p'}(P)=\|P\|_p$. Thus the degree-one case follows directly.}  If $m=2$, the only multiplicity
patterns are $(2)$ and $(1,1)$.  For $\kappa=(1,1)$,
\[
 r=2,\qquad N_{(1,1)}=2,\qquad {\color{black}\pi_{(1,1)}}=\frac12,
\]
and \Cref{lem:one-pattern-HL} gives
\[
 A_{(1,1)}(P)^q
 \le2^{q/2}2^{2q/p}2^{1-\eta}\|P\|_p^q
 \le4\sqrt2\,\|P\|_p^q.
\]
The pattern $(2)$ is controlled by \eqref{eq:principal-HL}.  Since the two
coefficient sets are disjoint, $\Lambda_{2,q}(P)^q\le C_2\|P\|_p^q$ for
an absolute constant $C_2$.  Enlarging the constant in
\eqref{eq:main-A-contractivity} once proves the assertion for $m=2$ and
completes the proof.
}

\section{{\color{black}Wiener slices and exact Bombieri $q$-norm contractivity}}\label{sec:wiener-slices}

{\color{black}We keep $\Lambda_{m,q}$ for the classical Bombieri $q$-norm recalled in the introduction.  In this section $\|P\|$ means $\|P\|_\infty$.  When a coordinate is singled out, the notation $P_d$ will always denote the part of transverse degree $d$ in the remaining variables.}

\medskip
\begin{proposition}[{\color{black}Endpoint Bombieri $q$-norm estimate}]
\label{prop:orbit-coefficient-input}
{\color{black}
{\color{black}For the duration of this section, set $q_d:=2d/(d+1)$ for $d\ge1$.}  Then there is an absolute constant $C_0>0$ such that
every complex $d$-homogeneous polynomial $P:\CC^n\to\CC$ satisfies
\begin{equation}\label{eq:beta-asymptotic-input}
 \Lambda_{d,q_d}(P)
 \le \left(1+\frac{C_0}{d}\right)\|P\|,
 \qquad d\ge1.
\end{equation}
}
\end{proposition}

\begin{proof}
{\color{black}\hypersetup{linkcolor=black}
This is \Cref{thm:main-A} with $m=d$ and $p=\infty$.
}
\end{proof}

{\color{black}\hypersetup{linkcolor=black}
Proposition~\ref{prop:orbit-coefficient-input} applies to the transverse
slices associated with a distinguished coordinate.  Their control rests on
the following classical coefficient estimate.
}
{\color{black}\hypersetup{citecolor=black}The first-coefficient estimate and its
transfer to arbitrary coefficients are recorded, respectively, in
\cite[Theorems~8 and~14]{BrevigGrepstadInstanesWiener}.}

\begin{lemma}[Wiener coefficient estimate]\label{lem:Wiener-classical}
Let
\[
 f:\DD\longrightarrow\CC,
 \qquad f(\zeta)=a_0+\sum_{k\ge1}a_k\zeta^k,
\]
be holomorphic and satisfy
\[
 |f(\zeta)|\le1\qquad\text{for every }\zeta\in\DD.
\]
Then
\[
 |a_k|\le1-|a_0|^2,\qquad k\ge1.
\]
\end{lemma}

For the homogeneous polynomials considered here, this gives the following
slice estimate.

\begin{lemma}\label{lem:Wiener-slice}
\textcolor{black}{Let $n\ge2$ and let} $P:\CC^n\to\CC$ be
$m$-homogeneous, with $\|P\|\le1$.  Writing
$z'=(z_2,\ldots,z_n)$, define the {\color{black}$d$-homogeneous} polynomials
${\color{black}P_d}:\CC^{n-1}\to\CC$ by
\begin{equation}\label{eq:Wiener-decomposition}
 P(z_1,z')=c z_1^m+
 {\color{black}\sum_{d=1}^m} z_1^{{\color{black}m-d}}{\color{black}P_d}(z').
\end{equation}
Then
\begin{equation}\label{eq:Wiener-slice}
 \|{\color{black}P_d}\|\le1-|c|^2,
 \qquad {\color{black}1\le d\le m}.
\end{equation}
\end{lemma}

\begin{proof}
Fix $u\in\TT^{n-1}$ and define
\[
 f_u:\DD\longrightarrow\CC,\qquad
 f_u(\zeta):=P(1,\zeta u)
 =c+{\color{black}\sum_{d=1}^mP_d(u)\zeta^d}.
\]
For $|\zeta|<1$ and $u\in\TT^{n-1}$,
\[
 \|(1,\zeta u)\|_\infty=1,
\]
so the definition of $\|P\|$ gives
\[
 |f_u(\zeta)|=|P(1,\zeta u)|\le\|P\|\le1.
\]
Thus the hypotheses of \Cref{lem:Wiener-classical} hold, and
\[
 |{\color{black}P_d}(u)|\le1-|c|^2,\qquad {\color{black}1\le d\le m}.
\]
Hence
\[
 \sup_{u\in\TT^{n-1}}|{\color{black}P_d}(u)|\le1-|c|^2.
\]
Since ${\color{black}P_d}$ is a polynomial, the maximum-modulus principle in each
coordinate gives
\[
 \|{\color{black}P_d}\|=
 \sup_{\|z'\|_\infty\le1}|{\color{black}P_d}(z')|
 =\sup_{u\in\TT^{n-1}}|{\color{black}P_d}(u)|
 \le1-|c|^2.
\]
\end{proof}

\medskip
The Wiener slice estimate controls the transverse layers near a dominant pure
power.  {\color{black}After separating the pure powers in the top slice, every
genuinely transverse layer retains the inverse-multinomial gain.  This is the
additional defect that confines the transition to a bounded window.}

\medskip

\begin{proof}[Proof of \Cref{thm:main-B}]
{\color{black}
Let $P(z)=\sum_{|\alpha|=m}b_\alpha z^\alpha$ be a complex
$m$-homogeneous polynomial on $\CC^n$, where $p\in[2m,\infty]$, and put
$q:=q_{m,p}$.  {\color{black}If $P=0$, the desired inequality is trivial.  If $n=1$, then
$P(z)=cz^m$ and $\Lambda_{m,q}(P)=|c|=\|P\|_p$, so exact contractivity
holds.  Hence assume $P\ne0$ and $n\ge2$,} and normalize $\|P\|_p=1$.
Throughout the proof, write
\[
 {\color{black}P_{\mathrm{diag}}}(z):=\sum_{j=1}^n b_{me_j}z_j^m.
\]
}

Set
\begin{equation}\label{eq:flat-uniform-constant}
 K_{\rm flat}:=
 \sup_{\substack{d,n\in\NN,\ {\color{black}u\in[2d,\infty]}\\
                   0\ne R:\CC^n\to\CC,\ R\ d\text{-homogeneous}}}
 \frac{\Lambda_{d,q_{d,u}}(R)}{\|R\|_u}.
\end{equation}
{\color{black}\hypersetup{linkcolor=black}
This number is finite, since $u\ge2d$ places $(d,u)$ in the range of
\Cref{thm:main-A}, which gives
\[
 \Lambda_{d,q_{d,u}}(R)
 \le\left(1+\frac Cd\right)\|R\|_u,
\]
for every $d\ge1$ and $u\ge2d$.  Hence $K_{\rm flat}\le1+C$.
Thus $K_{\rm flat}$ controls the orbit-norm comparisons in all transverse
degrees, uniformly in the dimension and the admissible domain exponent.
}

We treat the endpoint $p=\infty$ separately, since the weighted $p$-power
slice argument below applies only for finite $p$.  Here $q=q_m$ and the diagonal
projection satisfies
\[
 \|{\color{black}P_{\mathrm{diag}}}\|_\infty
 =\sum_j|b_{me_j}|\le1,
\]
by independent phase alignment.  Put $h:=\max_j|b_{me_j}|$ and fix
$h_0\in(2^{-1/2},1)$.  If $h\le h_0$, then
\[
 \sum_j|b_{me_j}|^q
 \le h^{q-1}\sum_j|b_{me_j}|
 \le h_0^{q-1},
\]
while \eqref{eq:off-diagonal-isotropic} tends to zero uniformly with $m$.
Thus $\Lambda_{m,q}(P)<1$ for all sufficiently large $m$ in this case.
If $h>h_0$, rotate and relabel coordinates so that
$b_{me_1}=h\ge0$, and write
\[
 P(z_1,z')=hz_1^m+\sum_{d=1}^m z_1^{m-d}P_d(z').
\]
With $\delta:=1-h^2$, \Cref{lem:Wiener-slice} gives
$\|P_d\|_\infty\le\delta$.  {\color{black}\hypersetup{linkcolor=black}
If $\delta=0$, every $P_d$ vanishes, so $P(z)=z_1^m$ and
$\Lambda_{m,q}(P)=1$.  We may therefore assume $\delta>0$.
Since $q_d\le q_m=q$, identity \eqref{eq:intro-orbit-identity},
\Cref{prop:orbit-coefficient-input}, and monotonicity of finite sequence
norms give}
\[
 \Lambda_{d,q}(P_d)
 \le\Lambda_{d,q_d}(P_d)
 \le K_{\rm flat}\delta.
\]
\textcolor{black}{For $m\ge3$ one has $q-1=(m-1)/(m+1)\ge1/2$; hence
\Cref{lem:inverse-multinomial}, applied to two parts, gives
\[
 \sum_{d=1}^m\binom md^{1-q}
 =1+\sum_{d=1}^{m-1}\binom md^{-(q-1)}
 \le1+K.
\]
{\color{black}The orbit decomposition is
\[
 \Lambda_{m,q}(P)^q
 =h^q+\sum_{d=1}^m\binom md^{1-q}\Lambda_{d,q}(P_d)^q.
\]
Using this identity and the preceding bound,}}
\[
 \Lambda_{m,q}(P)^q
 \le h^q+C\delta^q.
\]
Since $q-1=(m-1)/(m+1)\ge1/3$ for $m\ge2$, choosing $h_0$ sufficiently
close to one gives $C\delta^{q-1}\le1/4$ whenever $h>h_0$.  Hence
\[
 \Lambda_{m,q}(P)^q
 \le(1-\delta)^{q/2}+\frac{\delta}{4}.
\]
{\color{black}Here $0<q/2\le1$, so concavity of $t\mapsto t^{q/2}$ at $t=1$ gives
\[
 (1-\delta)^{q/2}\le1-\frac q2\delta.
\]
Moreover $q\ge1$, hence $q/2-1/4\ge1/4$. Therefore}
\[
 \Lambda_{m,q}(P)^q
 \le1-\left(\frac q2-\frac14\right)\delta
 \le1-\frac{\delta}{4}<1.
\]
{\color{black}\hypersetup{linkcolor=black}
This proves \Cref{thm:main-B} at $p=\infty$.  We henceforth assume
$p<\infty$.}

Let
\[
 h:=\max_j|b_{me_j}|,\qquad
 s_p:=\frac p{p-m}=\frac1{1-m/p}.
\]
{\color{black}Thus $s_p=s_{m,p}$ in the notation of the introduction.}
The coordinatewise-rotation projection in
\eqref{eq:diagonal-projection} gives
\[
 \|{\color{black}P_{\mathrm{diag}}}\|_p\le\|P\|_p=1.
\]

Moreover, putting $y_j=z_j^m$ gives
$\sum_j|y_j|^{p/m}=\sum_j|z_j|^p$, and therefore
\[
 \|{\color{black}P_{\mathrm{diag}}}\|_p
 =\sup_{\|y\|_{p/m}\le1}
   \left|\sum_{j=1}^n b_{me_j}y_j\right|
 =\|(b_{me_j})_j\|_{(p/m)'}
 =\|(b_{me_j})_j\|_{s_p}.
\]
Consequently
\[
 \|(b_{me_j})_j\|_{s_p}=\|{\color{black}P_{\mathrm{diag}}}\|_p\le1.
\]
{\color{black}
Introduce the exact distance from the critical exponent
\begin{equation}\label{eq:beta-definition-proof}
 \tau:=1-\frac{2m}{p},
 \qquad
 \beta:=\frac{q}{s_p}-1.
\end{equation}
Since
\[
 s_p=\frac{2}{1+\tau},
 \qquad
 q=\frac{2m}{m+\tau},
\]
we have
\begin{equation}\label{eq:beta-identities}
 \beta=\frac{(m-1)\tau}{m+\tau},
 \qquad q=s_p(1+\beta),
 \qquad q\left(1-\frac mp\right)=1+\beta.
\end{equation}
In particular, $0\le\beta<1$.  Fix $h_0\in(2^{-1/2},1)$.  If
$h\le h_0$, then
\[
 \sum_j|b_{me_j}|^q
 \le h^{q-s_p}\sum_j|b_{me_j}|^{s_p}
 \le h_0^{s_p\beta}
 \le h_0^\beta.
\]
}
By \eqref{eq:off-diagonal-isotropic},
\[
 \Lambda_{m,q}(P-{\color{black}P_{\mathrm{diag}}})^q\le C m^{-(q-1)}.
\]
{\color{black}For $m\ge3$,
\[
 q-1=\frac{m-1+2m/p}{m+1-2m/p}\ge\frac{m-1}{m+1}\ge\frac12
\]
and, more precisely, $m^{-(q-1)}\le C_1/m$ for an absolute $C_1$.
{\color{black}To see that the latter constant is uniform, write
\[
 m^{-(q-1)}=m^{-1}m^{1-(q-1)}.
\]
Since
\[
 1-(q-1)=2-q=
 \frac{2(1-2m/p)}{m+1-2m/p}
 \le\frac{2}{m},
\]
we have
\[
 m^{1-(q-1)}\le m^{2/m}\le \sup_{x\ge2}x^{2/x}<\infty.
\]
Thus one may take $C_1:=\sup_{x\ge2}x^{2/x}$.}
Since $\beta\mapsto h_0^\beta$ is convex on $[0,1]$,
{\color{black}its graph lies below the chord joining $(0,1)$ and $(1,h_0)$; hence}
\[
 h_0^\beta\le(1-\beta)\cdot1+\beta h_0
 =1-(1-h_0)\beta.
\]
Consequently $\Lambda_{m,q}(P)^q\le1$ whenever
\begin{equation}\label{eq:beta-flat-condition}
 \beta\ge\frac{B_1}{m}
\end{equation}
for a sufficiently large absolute $B_1$.  This settles the case
$h\le h_0$.}

Assume now that $h>h_0$.  Choose $j_0$ with $|b_{me_{j_0}}|=h$ and
permute the coordinates so that $j_0=1$.  Next choose $\omega_0\in\TT$
such that $\omega_0^m b_{me_1}=|b_{me_1}|=h$, and replace $P$ by
\[
 \widetilde P(z_1,z'):=P(\omega_0 z_1,z').
\]
The map $(z_1,z')\mapsto(\omega_0 z_1,z')$ preserves the $\ell_p$ ball,
so $\|\widetilde P\|_p=\|P\|_p$, and it changes each coefficient only
by a unimodular factor.  Thus all coefficient moduli, and hence
$\Lambda_{m,q}$, are unchanged.  We may therefore write
$b_{me_1}=h\ge0$.  Write
\[
 P(z_1,z')=hz_1^m+
 \sum_{d=1}^m z_1^{m-d}P_d(z'),
\]
and set $\delta=1-h^2<1/2$.  For $\|x\|_p\le1$ and $0<u<1$, the
function
\[
 \zeta\longmapsto
 P\bigl((1-u)^{1/p},u^{1/p}\zeta x\bigr)
\]
maps the unit disc into the closed unit disc, since for $|\zeta|<1$,
\[
 \bigl\|\bigl((1-u)^{1/p},u^{1/p}\zeta x\bigr)\bigr\|_p^p
 =(1-u)+u|\zeta|^p\|x\|_p^p\le1.
\]
From
\[
 P(z_1,z')=hz_1^m+\sum_{d=1}^m z_1^{m-d}P_d(z')
\]
and the $d$-homogeneity of $P_d$,
\[
 P_d(u^{1/p}\zeta x)=u^{d/p}\zeta^dP_d(x).
\]
Hence
\begin{align*}
&P\bigl((1-u)^{1/p},u^{1/p}\zeta x\bigr)\\
&\quad=h\bigl((1-u)^{1/p}\bigr)^m
 +\sum_{d=1}^m\bigl((1-u)^{1/p}\bigr)^{m-d}
   P_d(u^{1/p}\zeta x)\\
&\quad=h(1-u)^{m/p}
 +\sum_{d=1}^m(1-u)^{(m-d)/p}u^{d/p}P_d(x)\,\zeta^d.
\end{align*}
Its constant coefficient is $h(1-u)^{m/p}$.  Wiener's estimate applied
to the coefficient of $\zeta^d$ therefore yields
\[
 (1-u)^{(m-d)/p}u^{d/p}|P_d(x)|
 \le1-h^2(1-u)^{2m/p}.
\]
Taking the supremum over $\|x\|_p\le1$ gives
\begin{equation}\label{eq:weighted-Wiener-slice}
 \|P_d\|_p
 \le u^{-d/p}(1-u)^{-(m-d)/p}
 \left[1-h^2(1-u)^{2m/p}\right].
\end{equation}
{If $\delta>0$, choose $u=\delta$.  Since
$h^2=1-\delta$ and $2m/p\le1$,
\[
 1-h^2(1-\delta)^{2m/p}
 =1-(1-\delta)^{1+2m/p}
 \le(1+2m/p)\delta\le2\delta,
\]
where we used $1-(1-x)^\theta\le\theta x$ for $1\le\theta\le2$.
Also $(m-d)/p\le m/p\le1/2$ and, because $h>h_0>2^{-1/2}$,
$\delta=1-h^2<1/2$.  Hence
\[
 (1-\delta)^{-(m-d)/p}
 \le(1-\delta)^{-1/2}\le\sqrt2.
\]
Substituting these two bounds in \eqref{eq:weighted-Wiener-slice} gives
\begin{equation}\label{eq:weighted-Wiener-power}
 \|P_d\|_p\le2\sqrt2\,\delta^{1-d/p}.
\end{equation}}
\textcolor{black}{If $\delta=0$, the right-hand side of
\eqref{eq:weighted-Wiener-slice} is $O(u^{1-d/p})$ as $u\downarrow0$.
Since $d/p\le m/p<1$, it tends to zero, and hence every $P_d$ vanishes.}
We may therefore assume $\delta>0$.

\textcolor{black}{Since $p\ge2m\ge2d$ for $1\le d\le m$, all transverse
degrees are admissible.  For fixed $p$, the function}
\[
 d\longmapsto q_{d,p}=\frac{2dp}{dp+p-2d}
\]
is increasing on \textcolor{black}{$1\le d\le p/2$}; equivalently,
\[
 \frac1{q_{d,p}}=\frac12+\frac1{2d}-\frac1p
\]
decreases with $d$.  Therefore $q_{d,p}\le q_{m,p}=q$.  Since finite
sequence norms decrease when the exponent increases,
\[
 \Lambda_{d,q}(P_d)
 =\|a(\check P_d)\|_q
 \le\|a(\check P_d)\|_{q_{d,p}}
 =\Lambda_{d,q_{d,p}}(P_d).
\]
Since $p\ge2d$, the pair $(d,p)$ is admissible in the supremum defining
$K_{\rm flat}$.  Therefore
\[
 \Lambda_{d,q}(P_d)
 \le\Lambda_{d,q_{d,p}}(P_d)
 \le K_{\rm flat}\|P_d\|_p.
\]

{\color{black}
Keep every pure power in the diagonal mass. Write
\[
 {\color{black}(P_m)_{\mathrm{diag}}}(z'):=\sum_{j=2}^n b_{me_j}z_j^m,
 \qquad
 R_m:=P_m-{\color{black}(P_m)_{\mathrm{diag}}},
\]
and set
\[
 M_{\rm diag}:=\Lambda_{m,q}({\color{black}P_{\mathrm{diag}}})^q=\sum_{j=1}^n|b_{me_j}|^q,
 \qquad
 M_{\rm off}:=\Lambda_{m,q}(P-{\color{black}P_{\mathrm{diag}}})^q.
\]
The coefficient of $z_1^{m-d}(z')^\alpha$, $|\alpha|=d$, carries the
multinomial factor
\[
 \frac{m!}{(m-d)!\alpha!}=\binom md\frac{d!}{\alpha!}.
\]
Consequently the genuinely nonpure mass is exactly
\begin{equation}\label{eq:refined-transverse-decomposition}
 M_{\rm off}=
 \sum_{d=1}^{m-1}\binom md^{1-q}\Lambda_{d,q}(P_d)^q
 +\Lambda_{m,q}(R_m)^q.
\end{equation}
For $d<m$, \eqref{eq:weighted-Wiener-power} and the preceding application
of $K_{\rm flat}$ give
\[
 \Lambda_{d,q}(P_d)^q
 \le(2\sqrt2K_{\rm flat})^q\delta^{q(1-d/p)}.
\]
Since $0<\delta<1$ and $d\le m$, the power on the right is bounded by
$(2\sqrt2K_{\rm flat})^q\delta^{q(1-m/p)}$.  Also, with $\eta:=q-1$,
\Cref{lem:inverse-multinomial} gives
\begin{equation}\label{eq:refined-binomial-sum}
 \sum_{d=1}^{m-1}\binom md^{1-q}\le K m^{-\eta}.
\end{equation}
The polynomial $R_m$ is the nonpure part of the $m$-homogeneous
polynomial $P_m$.  Applying \eqref{eq:off-diagonal-isotropic} to $P_m$ and
then using \eqref{eq:weighted-Wiener-power} with $d=m$ yields
\[
 \Lambda_{m,q}(R_m)^q
 \le C m^{-\eta}\|P_m\|_p^q
 \le C m^{-\eta}\delta^{q(1-m/p)}.
\]
Combining these estimates with
\eqref{eq:refined-transverse-decomposition}, and using
\eqref{eq:beta-identities}, gives the refined transverse defect
\begin{equation}\label{eq:refined-transverse-defect}
 M_{\rm off}\le Lm^{-\eta}\delta^{q(1-m/p)}
   =Lm^{-\eta}\delta^{1+\beta},
\end{equation}
where $L$ is absolute.  Here the factor $m^{-\eta}$ is present also at the top layer $d=m$, since
the pure powers in that layer were removed before the off-diagonal estimate.

We compare this defect with the exact diagonal deficit. Put
\[
 x_j:=|b_{me_j}|^{s_p},
 \qquad H:=\max_jx_j=h^{s_p}=1-v.
\]
Since $\sum_jx_j\le1$ and $q=s_p(1+\beta)$,
\begin{equation}\label{eq:diagonal-beta-defect}
 M_{\rm diag}=\sum_jx_j^{1+\beta}
 \le H^\beta\sum_jx_j
 \le(1-v)^\beta
 \le1-\beta v.
\end{equation}
The last inequality is the tangent-line bound for the concave function
$t\mapsto t^\beta$.  Moreover,
\[
 v=1-(1-\delta)^{s_p/2}\ge\frac{\delta}{2},
\]
because $1\le s_p\le2$.  {\color{black}Indeed, with $a:=s_p/2\in[1/2,1]$, concavity of $t\mapsto t^a$ and the tangent line at $t=1$ give
\[
 (1-\delta)^a\le1-a\delta,
 \qquad
 v=1-(1-\delta)^a\ge a\delta\ge\frac{\delta}{2}.
\]}
Thus $\delta\le2v$.  Recalling that
$m^{-\eta}\le C_1/m$ and that $0<\delta<1$, \eqref{eq:refined-transverse-defect}
implies
\[
 M_{\rm off}\le\frac{C_3}{m}\delta^{1+\beta}
 \le\frac{2C_3}{m}v.
\]
Together with \eqref{eq:diagonal-beta-defect}, this proves $M_{\rm diag}+M_{\rm off}\le1$ whenever
\begin{equation}\label{eq:beta-dominant-condition}
 \beta\ge\frac{B_2}{m}
\end{equation}
for a sufficiently large absolute $B_2$.  Taking
$B:=\max\{B_1,B_2\}$ proves the quantitative criterion
\eqref{eq:beta-criterion-intro}. {\color{black}Enlarging $B_1$ and $B_2$ slightly if necessary, the two absorption inequalities above may be taken with strict margin; consequently, whenever $h<1$ the resulting estimate is strict.}

Finally, write {\color{black}$p=2m+t$}.  Since
\[
 1-\frac{2m}{p}=1-\frac{2m}{2m+t}=\frac{t}{2m+t}
\]
and
\[
 m+1-\frac{2m}{p}
 =m+1-\frac{2m}{2m+t}
 =\frac{2m^2+(m+1)t}{2m+t},
\]
we obtain
\begin{equation}\label{eq:beta-additive-window}
 \beta=
 {\color{black}\frac{(m-1)t}{2m^2+(m+1)t}}.
\end{equation}
{\color{black}For fixed $m$, differentiation gives
\[
 \frac{d}{dt}
 \frac{(m-1)t}{2m^2+(m+1)t}
 =\frac{2m^2(m-1)}{[2m^2+(m+1)t]^2}>0,
\]
so} the right-hand side is increasing in {\color{black}$t$}.  If $m\ge\max\{2,A\}$ and
{\color{black}$t\ge A$}, then
\[
 m\beta
 \ge\frac{m(m-1)A}{2m^2+(m+1)A}.
\]
{\color{black}Because $m\ge2$ gives $m-1\ge m/2$, while $A\le m$ gives
\[
 2m^2+(m+1)A\le2m^2+m(m+1)\le4m^2,
\]
we obtain}
\[
 m\beta\ge\frac{(m^2/2)A}{4m^2}=\frac A8.
\]
Taking $A=8B$ and enlarging $m_0$ to include the endpoint argument therefore
proves that $p\ge2m+A$ implies $\beta\ge B/m$, completing the proof of
\Cref{thm:main-B}.}
\end{proof}

\section{Proof of the Bombieri--Weyl theorem}
\label{sec:bombieri-weyl}

{\color{black}
\[
 \|P\|_{\mathrm{BW}}=\Lambda_{m,2}(P),
 \qquad
 \Lambda_{m,q}(P)=\|a(\check P)\|_q.
\]}

{\color{black}
{\color{black}At the Hilbertian exponent, the classical Bombieri $q$-norm becomes the Bombieri--Weyl norm:} \eqref{eq:orbit-norm-definition} gives
\[
 \Lambda_{m,2}(P)^2
 =\sum_{|\alpha|=m}\frac{\alpha!}{m!}|b_\alpha|^2
 =\|P\|_{\mathrm{BW}}^2.
\]
Thus
\begin{equation}\label{eq:BW-orbit-identification}
 \Lambda_{m,2}(P)=\|P\|_{\mathrm{BW}}.
\end{equation}
Through \eqref{eq:intro-orbit-identity}, the same identity reads
$\|a(\check P)\|_2=\|P\|_{\mathrm{BW}}$.
}

\begingroup
\begin{proof}[Proof of \Cref{thm:Bombieri-Weyl-threshold}]
For the necessity of the critical scale, take
$P_n(z)=z_1^m+\cdots+z_n^m$.  Then $\|P_n\|_{\mathrm{BW}}=\sqrt n$.  If
$p\ge m$, then H\"older's inequality gives $\|P_n\|_p=n^{1-m/p}$, and hence
\[
 \frac{\|P_n\|_{\mathrm{BW}}}{\|P_n\|_p}=n^{m/p-1/2}.
\]
This diverges with $n$ whenever $m\le p<2m$.  If $p<m$, already
$\|P_n\|_p=1$, so the same conclusion follows.  Thus no dimension-free
comparison is possible for $p<2m$.

{\color{black}\hypersetup{linkcolor=black}
Let $P$ be arbitrary.  At $p=2m$ we have $q_{m,2m}=2$, so
\Cref{thm:main-A} and \eqref{eq:BW-orbit-identification} give
\[
 \|P\|_{\mathrm{BW}}
 \le\left(1+\frac{C}{m}\right)\|P\|_{2m},
\]
{\color{black}which proves the critical estimate in \Cref{thm:Bombieri-Weyl-threshold}.}
}

{\color{black}\hypersetup{linkcolor=black}
Let $m\ge m_0$ and $p\ge2m+A$.  Then \Cref{thm:main-B} applied directly
to $P$ yields
\[
 \Lambda_{m,q_{m,p}}(P)\le\|P\|_p.
\]
Since $p>2m$, formula \eqref{eq:q-isotropic-full} gives $q_{m,p}<2$.
By \eqref{eq:intro-orbit-identity}, monotonicity of finite sequence norms
therefore gives
\[
 \|P\|_{\mathrm{BW}}
 =\Lambda_{m,2}(P)
 \le\Lambda_{m,q_{m,p}}(P)
 \le\|P\|_p,
\]
{\color{black}which proves the contractive assertion in \Cref{thm:Bombieri-Weyl-threshold}.}
{\color{black}If $P\ne0$ and equality holds, then the two extreme terms in
\[
 \|a(\check P)\|_2\le\|a(\check P)\|_{q_{m,p}}\le\|P\|_p
\]
are equal.  Since $q_{m,p}<2$, equality in the first inequality forces the
ordered coefficient array $a(\check P)$ to have exactly one nonzero
coordinate.  {\color{black}Indeed, writing $x=a(\check P)$ and normalizing $\|x\|_{q_{m,p}}=1$, we have $|x_\nu|\le1$ for every $\nu$, and therefore
\[
 \|x\|_2^2=\sum_\nu|x_\nu|^2
 \le\sum_\nu|x_\nu|^{q_{m,p}}=1.
\]
Because $q_{m,p}<2$, equality can hold only when every $|x_\nu|$ is either $0$ or $1$; the normalization then leaves exactly one nonzero coordinate.}
Its symmetry then forces that coordinate to be
$(j,\ldots,j)$.  {\color{black}For if a nonzero ordered index were not diagonal, its orbit under permutations would contain at least two distinct ordered indices, and symmetry of $\check P$ would give the same nonzero coefficient at both.}
Hence $P(z)=c z_j^m$.  Conversely, every such coordinate
pure power attains equality.}  Finally, $P(z)=z_1^m$ has
\(\|P\|_{\mathrm{BW}}=\|P\|_p=1\), so the constant one is optimal.
}
\end{proof}
\endgroup

{\color{black}The bounded-window supercritical regime} also determines the equality cases.

\begin{corollary}[{\color{black}Quantitative stability beyond the bounded critical window}]
\label{cor:BW-rigidity}
{\color{black}Let $A$ and $m_0$ be as in \Cref{thm:main-B}, let
$m\ge m_0$ and $p\ge2m+A$,} {\color{black}and let $P\in\mathcal P_m(\CC^n)$ satisfy $\|P\|_p=1$.} Put
$q=q_{m,p}$ and $\delta:=1-\|P\|_{\mathrm{BW}}$.  Then
\begin{equation}\label{eq:BW-rigidity-max}
 \max_{i_1,\ldots,i_m}
 |\check P(e_{i_1},\ldots,e_{i_m})|
 \ge (1-\delta)^{2/(2-q)}.
\end{equation}
If the right-hand side of \eqref{eq:BW-rigidity-max} is larger than
$2^{-1/q}$, the maximizing ordered coefficient is necessarily a pure-power
coefficient.  In that case, for some $j$,
\begin{equation}\label{eq:BW-rigidity-distance}
 \|P-b_{me_j}z_j^m\|_{\mathrm{BW}}
 \le
 \left(1-(1-\delta)^{4/(2-q)}\right)^{1/2}.
\end{equation}
\end{corollary}

\begin{proof}
{\color{black}\hypersetup{linkcolor=black}
Let $x=a(\check P)$, {\color{black}so that
$x_{i_1,\ldots,i_m}=\check P(e_{i_1},\ldots,e_{i_m})$.}  In the stated range,
\Cref{thm:main-B}, \eqref{eq:intro-orbit-identity}, and
\eqref{eq:BW-orbit-identification} give
\[
 \|x\|_2=\|P\|_{\mathrm{BW}}
 \le \|x\|_q\le\|P\|_p,
 \qquad q=q_{m,p}<2.
\]
}
For the quantitative statement, let
$M:=\max_{i_1,\ldots,i_m}|x_{i_1,\ldots,i_m}|$.  Since $\|x\|_q\le1$,
\[
 (1-\delta)^2\le\|x\|_2^2
 =\sum |x_i|^{2-q}|x_i|^q
 \le M^{2-q}\|x\|_q^q
 \le M^{2-q},
\]
which proves \eqref{eq:BW-rigidity-max}.  If a maximizing coefficient belongs
to a nonpure permutation orbit, that orbit contains at least two equal
coefficients, and therefore $1\ge\|x\|_q^q\ge2M^q$.  Thus
$M\le2^{-1/q}$, proving the second assertion.  If the maximizing coefficient
is the pure coefficient $b_{me_j}$, orthogonality of the ordered coefficient
coordinates and \eqref{eq:BW-orbit-identification} give
\[
 \|P-b_{me_j}z_j^m\|_{\mathrm{BW}}^2
 =\|x\|_2^2-|b_{me_j}|^2
 \le1-M^2,
\]
and \eqref{eq:BW-rigidity-max} yields
\eqref{eq:BW-rigidity-distance}.
\end{proof}

\section{The real case}
\label{sec:real-case}

{\color{black}In this section, $\|P\|_p$ denotes the supremum norm on the unit ball of $\ell_p^n(\mathbb R)$.}

{\color{black}The complex field is essential for the contractive phenomenon proved above.
The following elementary two-variable family shows that the Bombieri--Weyl comparison fails
to be contractive over the reals even at $p=\infty$.}

\begin{proposition}[{\color{black}Real obstruction}]
\label{prop:real-BW-obstruction}
{\color{black}Let $m=2k\ge2$ and
\[
 P_m(x,y)=(x^2-y^2)^k.
\]
Then, for every $p\in[1,\infty]$,
\[
 \|P_m\|_p=1,
 \qquad
 \|P_m\|_{\mathrm{BW}}^2
 =\frac{4^k}{\binom{2k}{k}}>1.
\]
Consequently the real Bombieri--Weyl comparison is never contractive in even
degree, for any $p$, and
\[
 \|P_m\|_{\mathrm{BW}}\sim(\pi k)^{1/4}.
\]}
\end{proposition}

\begin{proof}
{\color{black}Since $|x|,|y|\le1$ on the unit ball of $\ell_p^2(\mathbb R)$,
\[
 |P_m(x,y)|=|x^2-y^2|^k\le1,
\]
and equality holds at $(1,0)$.  Hence \(\|P_m\|_p=1\).
Expanding,
\[
 P_m(x,y)=\sum_{j=0}^k(-1)^j\binom{k}{j}x^{2k-2j}y^{2j},
\]
so the definition of the Bombieri--Weyl norm gives
\[
 \|P_m\|_{\mathrm{BW}}^2
 =\sum_{j=0}^k\frac{\binom{k}{j}^2}{\binom{2k}{2j}}.
\]
For each $j$,
\[
 \frac{\binom{k}{j}^2}{\binom{2k}{2j}}
 =\frac{\binom{2j}{j}\binom{2k-2j}{k-j}}{\binom{2k}{k}}.
\]
Therefore the central-binomial convolution yields
\[
 \|P_m\|_{\mathrm{BW}}^2
 =\frac1{\binom{2k}{k}}
   \sum_{j=0}^k\binom{2j}{j}\binom{2k-2j}{k-j}
 =\frac{4^k}{\binom{2k}{k}}.
\]
The strict inequality follows from \(\binom{2k}{k}<4^k\), while
\(\binom{2k}{k}\sim4^k/\sqrt{\pi k}\) gives the final asymptotic formula.}
\end{proof}

{\color{black}\noindent It is natural to ask whether an analogous obstruction persists for the real symmetric multilinear Bohnenblust--Hille inequality. Denote by $B^{\mathrm{sym}}_{\mathbb R,m}$ its optimal constant in degree $m$. In degree two take
\[
 T_2(x,y)=\frac12\bigl(x_0y_0+x_0y_1+x_1y_0-x_1y_1\bigr).
\]
Thus $T_2$ is symmetric, $\|T_2\|=1$, and its four coefficients have modulus $1/2$, giving the Bohnenblust--Hille quotient $\sqrt2$.  In degree four define directly
\[
 T_4(e_{i_1},e_{i_2},e_{i_3},e_{i_4})
 =\frac14(-1)^{1+\binom{i_1+i_2+i_3+i_4+1}{2}},
 \qquad i_1,i_2,i_3,i_4\in\{0,1\}.
\]
For weights $0,1,2,3,4$ the signs are respectively $-,+,+,-,-$.  Again $\|T_4\|=1$ and all sixteen coefficients have modulus $1/4$, so the quotient is again $\sqrt2$.  The same construction, using the classical Walsh system; see, for instance, \cite{FineWalsh}, gives in every even degree a symmetric $m$-linear form $T_m:(\ell_\infty^2)^m\to\mathbb R$ with $\|T_m\|=1$ and whose $2^m$ ordered coefficients all have modulus $2^{-m/2}$.  Therefore, with $q_m=2m/(m+1)$,
\[
\left(2^m(2^{-m/2})^{q_m}\right)^{1/q_m}=\sqrt2,
\]
and consequently
\[
 B^{\mathrm{sym}}_{\mathbb R,m}\ge\sqrt2\qquad(m\ \text{even}).
\]
This symmetric construction is related to the real lower-bound mechanism in \cite{DinizMunozPellegrinoSeoane}.}

\section*{Acknowledgments}

\noindent\textit{Funding.}
The research of D. N\'u\~nez-Alarc\'on, D. M. Pellegrino, and
A. Raposo Jr. was supported in part by CNPq Grants 406457/2023-9
(CNPq/MCTI Call 10/2023) and 403964/2024-5
(MCTI/CNPq Call 16/2024). D. M. Pellegrino was also supported by CNPq
Grant 305807/2025-0, and A. Raposo Jr. by CNPq Grant 302341/2025-0.
E. V. Teixeira acknowledges support from the Grayce B. Kerr Chair at
Oklahoma State University and partial support from the DARPA ExpMath
program under Agreement No.~HR0011262E029.

\medskip

\noindent\textit{AI-assisted verification.}
{\color{black}AI-assisted tools were used for exploratory calculations, consistency
checks, and editorial assistance.  The manuscript was written and verified by
the authors, who retain full responsibility for its mathematical content.}

\end{document}